\documentclass[11pt]{amsart}
\usepackage[T1]{fontenc}
\usepackage{lmodern}
\usepackage[margin=1.15in]{geometry}
\usepackage{amsmath,amssymb,amsthm,mathtools,bm,mathrsfs}
\usepackage{enumitem,booktabs,hyperref}
\usepackage{xcolor}
\numberwithin{equation}{section}
\allowdisplaybreaks[2]
\newtheorem{theorem}{Theorem}[section]
\newtheorem{thmA}{Theorem}[section]

\newtheorem{lemma}[theorem]{Lemma}
\newtheorem{corollary}[theorem]{Corollary}
\newtheorem{example}[theorem]{Example}
\newtheorem*{cconj}{Chern Conjecture}
\newtheorem*{luconj}{Lu's second gap conjecture}
\theoremstyle{remark}
\newtheorem{remark}[theorem]{Remark}
\newcommand{\R}{\mathbb R}
\newcommand{\Sph}{\mathbb S}
\newcommand{\tr}{\operatorname{tr}}
\newcommand{\diag}{\operatorname{diag}}
\newcommand{\rank}{\operatorname{rank}}
\newcommand{\Span}{\operatorname{span}}
\newcommand{\dd}{\mathrm d}
\newcommand{\G}{\mathcal G}
\title[Codimension Reduction in Spheres]{Codimension Reduction of Minimal Submanifolds with Flat Normal Bundle in Spheres}
\author{Juan Li}
\address{School of Mathematics, Hangzhou Normal University, Hangzhou, China}
\email{juanli@zju.edu.cn}

\author{Zhiyuan Xu}
\address{School of Mathematics, Hangzhou Normal University, Hangzhou, China}
\email{xuzy@hznu.edu.cn}
\thanks{Zhiyuan Xu is the corresponding author.}
\begin{document}

\begin{abstract}

Let $x:M^n\to\Sph^{n+q}$ be a minimal immersion of a complete connected manifold with flat normal bundle, where $n\ge 3$ and $q\ge 2$. We prove that if the squared norm $S$ of the second fundamental form satisfies $S\le 2n$, then either $x(M)$ is contained in a fixed totally geodesic subsphere $\Sph^{n+1}$, or $x:M\to x(M)$ is a Riemannian covering onto a standard minimal product of three spheres with $S\equiv 2n$. We also prove that $x(M)$ is contained in a fixed totally geodesic subsphere $\Sph^{n+1}$ under the assumption $S+\lambda_2\le 2n+\frac{2n}{3n-2}$, where $\lambda_2$ is the second largest eigenvalue of the Gram matrix of the shape operators.

\end{abstract}

\maketitle

\section{Introduction}

Let $x:M^n\to\Sph^{n+q}$ be a minimal immersion of a closed connected Riemannian manifold into the unit sphere. Denote by $S$ the squared norm of the second fundamental form of $M$. Through the Gauss equation, $S$ is directly related to the scalar curvature of the submanifold. The investigation into the gap phenomena of $S$ began with the seminal work of Simons \cite{Simons1968}, who proved that the pointwise bound $0\le S\le n/(2-\frac{1}{q})$ forces $S$ to be identically zero or identically equal to the upper endpoint. The equality case $S\equiv n/(2-\frac{1}{q})$ was subsequently classified by Chern--do Carmo--Kobayashi \cite{CDK1970}, and independently by Lawson \cite{Lawson1969} in the hypersurface case. Specifically, a closed minimal hypersurface satisfying $0 \le S \le n$ must be either totally geodesic or a Clifford minimal hypersurface. These results inspired a series of profound conjectures concerning the structure of minimal hypersurfaces with constant scalar curvature. S.S. Chern \cite{Chern1968} proposed the original and most influential conjecture, asserting that the possible values of the scalar curvature form a discrete set. This conjecture was later refined by Verstraelen--Montiel--Ros--Urbano \cite{Verstraelen1986}. Subsequently, Xu--Xu \cite{XX2017} formulated the stronger version. For clarity, we summarize them as follows.
\begin{cconj}
Let $M$ be a closed minimal hypersurface in $\Sph^{n+1}$.\\
{\rm (A)(Standard version)}  If $M$ has constant scalar curvature, then the possible values of the scalar curvature of $M$ form a discrete set. \\
{\rm (B)(Refined version)} If $M$ has constant scalar curvature, then $M$ is isoparametric. \\
{\rm (C)(Stronger version)} Set $a_k=(k-sgn(5-k))n$, for $1\le k\le5$. Then we have:
\begin{enumerate}[label=\textup{(\roman*)}]
\item For any fixed $1\le k\le4$, if $a_k\le S\le a_{k+1}$, then $M$ is isoparametric, and $S\equiv a_k$ or $S\equiv a_{k+1}$.
\item If $S\ge a_{5}$, then $M$ is isoparametric, and $S\equiv a_5$.
\end{enumerate}
\end{cconj}

Beyond the first pinching interval, the second gap problem asks whether a constant $S$ satisfying $n\le S\le2n$ must equal $n$ or $2n$. The corresponding second pinching problem asks whether the same rigidity holds under the pointwise condition, without assuming that $S$ is constant. Peng--Terng \cite{PTConstant1983} first established the second gap theorem for constant $S$ with a gap of $1/(12n)$, which was later improved by Yang--Cheng \cite{YC1998}, Suh--Yang \cite{SY2007}, and more recently by Chen \cite{Chen2026} and Cheng--Li--Wei--Yin \cite{CLWY2026}. The study of the second pinching problem was also initiated by Peng--Terng \cite{PTScalar1983} in low dimensions, with subsequent higher-dimensional estimates by Ding--Xin \cite{DX2011} and refinements by Lei, Xu and the second author \cite{LXXPublic,XX2017}. Chang \cite{Chang1993} proved Chern Conjecture (B) in dimension three. Deng--Gu--Wei \cite{DGW2017} verified Chern Conjecture (B) in dimension four under the Willmore assumption, which was later weakened to the constant cubic trace assumption by He--Xu--Zhao \cite{HXZ2026}. For higher-dimensional minimal hypersurfaces with constant scalar curvature and constant cubic trace, Cheng--Wei--Yamashiro \cite{CWY2025} obtained an improved second gap for complete hypersurfaces with $n\ge5$; Tan--Tang--Xie--Yan \cite{TTXY2026} proved local finiteness of the possible values of $S$ for closed embedded hypersurfaces; Ge--Tan--Yan--Zhang \cite{GTYZ2026} and Guan \cite{Guan2026} independently established the sharp second gap for closed minimal hypersurfaces with $n\ge5$.

In higher codimension, the geometry also depends on how the second fundamental form is distributed among the normal directions. Under weaker topological assumptions, Magliaro--Mari--Roing--Savas-Halilaj \cite{MMRS2024} proved that the classical rigidity conclusion under $S\le n/(2-\frac{1}{q})$ remains valid for complete minimal submanifolds. Ge--Li--Zhang \cite{GLZNormal2026} proved that, under the pointwise bound $0\le S\le n$ together with the flat normal bundle assumption, 
$M$ is either totally geodesic or a Clifford minimal submanifold. In a subsequent paper \cite{GLZFlat2026}, they established an explicit second-gap theorem under the additional assumption that $S$ is constant. Firester--Tsiamis \cite{FT2026} constructed a countable family of closed embedded minimal submanifolds $\{M^{n}_{j}\}$ in $\Sph^{n+q}$ with non-flat normal bundle whose constant values $S_j$ are dense in an interval, thereby disproving the Chern conjecture in higher codimension.

For every $n\geq3$, the following example gives an $n$-dimensional minimal submanifold of codimension two with flat normal bundle and $S\equiv2n$.
\begin{example}\label{example:3-sphere}
For positive integers $n_1,n_2,n_3$ with $n_1+n_2+n_3=n$, set $r_i=\sqrt{\frac{n_i}{n}}$ and consider
\[
 \prod_{i=1}^3\Sph^{n_i}(r_i)\subset\Sph^{n+2}\subset\Sph^{n+q}.
\]
Let $u_i$ be the unit radial vector in the $i$-th Euclidean factor. A normal vector in the sphere has the form $\nu=\sum_i c_i u_i$ with $\sum_i r_i c_i=0$, and its shape operator $A^\nu$ acts on the $i$-th factor as multiplication by $-\frac{c_i}{r_i}$. Hence
\[
 \tr A^\nu=-n\sum_i r_i c_i=0,\qquad
 |A^\nu|^2=n\sum_i c_i^2=n|\nu|^2.
\]
The product is therefore minimal, and the commutativity of its shape operators implies that the normal bundle is flat. The normal space has dimension two, which means the shape operators span a two-dimensional space, hence the norm identity gives $S=2n$. Thus the image cannot lie in a totally geodesic $\Sph^{n+1}$.
\end{example}

In this paper, we first prove the following optimal pinching theorem for complete minimal submanifolds with flat normal bundle.
\begin{thmA}\label{thm:main}
Let $x:M^n\to\Sph^{n+q}$ be a minimal immersion of a complete connected manifold with flat normal bundle, where $n\ge3$ and $q\ge2$. If 
$$S\le2n,$$ 
then one of the following holds:
\begin{enumerate}[label=\textup{(\roman*)}]
\item $x$ maps $M$ into a fixed totally geodesic subsphere $\Sph^{n+1}\subset\Sph^{n+q}$.
\item $S\equiv2n$, $\nabla A\equiv0$, and $x:M\to x(M)$ is a Riemannian covering with image congruent to the product $\Sph^{n_1}\!\left(\sqrt{\frac{n_1}{n}}\right)\times\Sph^{n_2}\!\left(\sqrt{\frac{n_2}{n}}\right)\times\Sph^{n_3}\!\left(\sqrt{\frac{n_3}{n}}\right)$ with $n_1+n_2+n_3=n$.
\end{enumerate}
\end{thmA}

As a corollary, we have the following codimension reduction theorem.
\begin{corollary}\label{cor:strict-pinching}
Let $x:M^n\to\Sph^{n+q}$ be a minimal immersion of a complete connected manifold with flat normal bundle, where $n\ge3$ and $q\ge2$. If the squared norm of the second fundamental form satisfies
\begin{equation}\label{eq:main-pinching}
 S<2n,
\end{equation}
then $x(M)$ is contained in a fixed totally geodesic subsphere $\Sph^{n+1}\subset\Sph^{n+q}$.
\end{corollary}

\begin{remark}\label{rem:surface-case}
For $n=2$, Wang \cite[Corollary~3.3]{Wang2013} proved that a connected minimal surface with flat normal bundle in a sphere is contained in a totally geodesic $\Sph^3$, without any curvature pinching assumption. This also follows from Lemma~\ref{lem:codimension-reduction}, since commuting trace-free symmetric $2\times2$ matrices span a space of dimension at most one. Under the hypotheses of Theorem~\ref{thm:main}, if $S<2n$ at one point, then alternative~\textup{(i)} holds. In alternative~\textup{(ii)}, the covering is finite if $M$ is closed. Under completeness alone, it may be infinite when one of the factors is a circle. If $n_1,n_2,n_3\ge2$, the product is simply connected, so the covering is an isometry. We recently learned that Zhang--Zhao \cite{ZZ2026} independently obtained the conclusion of Corollary~\ref{cor:strict-pinching} for closed minimal submanifolds with flat normal bundle under the stronger pinching assumption $n\le S\le14n/13$.
\end{remark}
Combining Corollary~\ref{cor:strict-pinching} with the results on Chern Conjecture for minimal hypersurfaces in \cite{CLWY2026,LXXPublic,PTConstant1983}, we obtain the following corollary.
\begin{corollary}
Let $x:M^n\to\Sph^{n+q}$ be a minimal immersion of a closed connected manifold with flat normal bundle.
\begin{enumerate}[label=\textup{(\roman*)}]
\item For $n=3$, if $S$ is constant and satisfies $S<6$, then $x(M)$ is either a totally geodesic sphere $\Sph^3$ or a Clifford minimal hypersurface.
\item For $n\ge4$, if $S$ is constant and satisfies $0\le S-n\le\frac{1}{2}n$, then $S=n$ and $x(M)$ is a Clifford minimal hypersurface.
\item For $n\ge3$, if $0\le S-n\le \frac{1}{18}n$, then $S\equiv n$ and $x(M)$ is a Clifford minimal hypersurface.
\end{enumerate}

In each case, $x(M)$ is contained in a fixed totally geodesic subsphere $\Sph^{n+1}\subset\Sph^{n+q}$, and $x:M\to x(M)$ is a finite Riemannian covering.
\end{corollary}

The pinching constant in the Simons rigidity theorem was improved by Li--Li \cite{LL1992} to $2n/3$ in codimension $q\ge 3$. Recent improvements of this bound for $n\ge3$ were obtained in \cite{GJL2026,Lei2026,LZ2026,XZ2026}. For $q\ge2$, denote by $\lambda_2$ the second largest eigenvalue of $\bigl(\langle A^\alpha,A^\beta\rangle\bigr)_{\alpha,\beta=1}^q$, the Gram matrix of the shape operators. Lu \cite{Lu2011} proved that if a closed connected minimal submanifold of $\Sph^{n+q}$ satisfies
\[S+\lambda_2\le n,\]
then it is either a totally geodesic $\Sph^{n}$, a Clifford minimal hypersurface in $\Sph^{n+1}$, or a Veronese surface in $\Sph^4$. In the latter two cases, $S+\lambda_2\equiv n$. Since $\lambda_2\le\frac{S}{2}$, Lu's condition $S+\lambda_2\le n$ is weaker than Li--Li's. Lu also proposed the following second gap conjecture.
\begin{luconj}
For each fixed $n$ and $q$, there exists $\varepsilon(n,q)>0$ such that any closed minimal submanifold of $\Sph^{n+q}$ whose $S+\lambda_2$ is constant and greater than $n$ satisfies
\[
 S+\lambda_2>n+\varepsilon(n,q).
\]
\end{luconj}

Ge--Li--Zhang \cite{GLZLu2026} settled this conjecture for closed minimal surfaces in $\Sph^4$ by proving that if $S+\lambda_2$ is constant, then it equals $0$ or $2$. The value $0$ corresponds to a totally geodesic $\Sph^2$, and the value $2$ to a Clifford torus in a totally geodesic $\Sph^3$ or to the Veronese surface in $\Sph^4$. Li--Zhao \cite{LZLu2026} disproved this conjecture for minimal surfaces in every codimension $q\ge3$ by constructing examples with non-flat normal bundle. Recently, Ding--Li--Yang--Zhang \cite{DLYZ2026} constructed examples with non-flat normal bundle whose constant values $S+\lambda_2$ are dense in $(n, 2n)$ for $n\ge3$ and $q\ge n+1$, and disproved this conjecture.

We next establish a codimension reduction theorem under a pointwise condition on $S+\lambda_2$, without assuming that $S$ or $S+\lambda_2$ is constant.
\begin{thmA}\label{thm:lu-pinching}
Let $x:M^n\to\Sph^{n+q}$ be a minimal immersion of a complete connected manifold with flat normal bundle, where $n\ge3$ and $q\ge2$. If
\begin{equation}\label{eq:lu-pinching}
 S+\lambda_2\le 2n+\frac{2n}{3n-2}
\end{equation}
then $x(M)$ is contained in a fixed totally geodesic subsphere $\Sph^{n+1}\subset\Sph^{n+q}$.
\end{thmA}

Combining Theorem~\ref{thm:lu-pinching} with the results on Chern Conjecture for minimal hypersurfaces in \cite{Chang1993,CLWY2026,dAB1990,LXXPublic}, we obtain the following corollary.
\begin{corollary}
Let $x:M^n\to\Sph^{n+q}$ be a minimal immersion of a closed connected manifold with flat normal bundle, where $q\ge2$.
\begin{enumerate}[label=\textup{(\roman*)}]
\item For $n=3$, if $S+\lambda_2$ is constant and satisfies $S+\lambda_2\le \frac{48}{7}$, then $x(M)$ is a totally geodesic sphere $\Sph^3$, a Clifford minimal hypersurface, or Cartan's minimal hypersurface.
\item For $n\ge4$, if $S+\lambda_2$ is constant and satisfies $0\le S+\lambda_2-n\le\frac{1}{2}n$, then $S+\lambda_2=n$ and $x(M)$ is a Clifford minimal hypersurface.
\item For $n\ge3$, if $0\le S+\lambda_2-n\le \frac{1}{18}n$, then $S\equiv n$, $\lambda_2\equiv0$, and $x(M)$ is a Clifford minimal hypersurface.
\end{enumerate}

In each case, $x(M)$ is contained in a fixed totally geodesic subsphere $\Sph^{n+1}\subset\Sph^{n+q}$, and $x:M\to x(M)$ is a finite Riemannian covering.
\end{corollary}

The paper is organized as follows. Section~\ref{sec:notation} introduces the notation and establishes the differential identities. Section~\ref{sec:matrix} proves two weighted matrix inequalities used to control the derivative terms in these identities. Section~\ref{sec:estimates} uses these inequalities to derive the differential estimates needed for the maximum principle arguments. The compactness result established there allows us to pass to complete pointed limits on which the relevant suprema are attained. The spectral estimate yields a contradiction with volume comparison on the limit in the proof of Theorem~\ref{thm:main}. Section~\ref{sec:main-proofs} establishes a codimension reduction criterion and combines these results to prove Theorems~\ref{thm:main} and~\ref{thm:lu-pinching}.

\section{Preliminaries and auxiliary identities}\label{sec:notation}

Let $x:M^n\to\Sph^{n+q}$ be a minimal immersion. We denote by $T_pM$ and $N_pM$ the tangent space at $p\in M$ and the normal space of the immersion in $\Sph^{n+q}$, respectively. The ambient metric and the induced inner products on the tangent, normal, and tensor bundles are denoted by $\langle\cdot,\cdot\rangle$, with corresponding norms $|\cdot|$.

We denote the Euclidean connection on $\R^{n+q+1}$ by $\widetilde\nabla$ and the Levi--Civita connection of $\Sph^{n+q}$ by $\bar\nabla$. Let $\nabla$ and $\nabla^\perp$ be the induced tangent and normal connections. The induced connections on tensor bundles are also denoted by $\nabla$, and their rough Laplacians by $\Delta$. For functions, $\Delta f=\operatorname{div}(\nabla f)$.

Let $\{e_i\}_{i=1}^n$ and $\{\nu_\alpha\}_{\alpha=1}^q$ be local orthonormal frames of the tangent and normal bundles, respectively, and let $\{\theta_i\}_{i=1}^n$ be the dual tangent coframe. Write the second fundamental form as
\[
 A=\sum_{\alpha=1}^q A^\alpha\nu_\alpha,\qquad
 A^\alpha=\sum_{i,j=1}^n A^\alpha_{ij}\,\theta_i\otimes\theta_j,
\]
where $A^\alpha_{ij}=\langle\bar\nabla_{e_i}e_j,\nu_\alpha\rangle=A^\alpha_{ji}$. The corresponding shape operator is also denoted by $A^\alpha$. Denote the squared of the second fundamental form as
\[
 S=|A|^2=\sum_{\alpha=1}^q|A^\alpha|^2
   =\sum_{\alpha=1}^q\sum_{i,j=1}^n(A^\alpha_{ij})^2.
\]

The connection forms are given by
\[
 \theta_{ij}(X)=\langle\nabla_Xe_i,e_j\rangle,\qquad
 \omega^\perp_{\alpha\beta}(X)
   =\langle\nabla_X^\perp\nu_\alpha,\nu_\beta\rangle.
\]
The components of $\nabla A$ satisfy
\[
 \sum_{k=1}^n A^\alpha_{ij;k}\theta_k
 =\dd A^\alpha_{ij}
   +\sum_{l =1}^n A^\alpha_{l  j}\theta_{l  i}
   +\sum_{l =1}^n A^\alpha_{il }\theta_{l  j}
   +\sum_{\beta=1}^q A^\beta_{ij}\omega^\perp_{\beta\alpha}.
\]
We denote the $\nu_\alpha$-component of $\nabla A$ by $\nabla A^\alpha$. Thus
\[
 |\nabla A^\alpha|^2=\sum_{i,j,k=1}^n(A^\alpha_{ij;k})^2,\qquad
 |\nabla A|^2=\sum_{\alpha=1}^q\sum_{i,j,k=1}^n(A^\alpha_{ij;k})^2.
\]
Since $x$ is minimal, we have $\tr A^\alpha=0$ and hence
\begin{equation}\label{eq:trace-derivative}
 \sum_{i=1}^n A^\alpha_{ii;k}=0.
\end{equation}
The symmetry of $A$ and the Codazzi equation give
\begin{equation}\label{eq:codazzi}
 A^\alpha_{ij;k}=A^\alpha_{ji;k}=A^\alpha_{ik;j}.
\end{equation}

Let $R$ and $R^\perp$ denote the curvature tensors of $\nabla$ and $\nabla^\perp$, respectively. The Gauss and Ricci equations are
\begin{equation}\label{eq:gauss}
 \langle R(e_i,e_j)e_l,e_k\rangle
 =\delta_{ik}\delta_{jl}-\delta_{il}\delta_{jk}
  +\sum_{\alpha=1}^q\left(A^\alpha_{ik}A^\alpha_{jl}-A^\alpha_{il}A^\alpha_{jk}\right)
\end{equation}
and
\begin{equation}\label{eq:ricci}
 \langle R^\perp(e_i,e_j)\nu_\beta,\nu_\alpha\rangle
 =\sum_{k=1}^n\left(A^\alpha_{ik}A^\beta_{jk}-A^\alpha_{jk}A^\beta_{ik}\right),
\end{equation}
where $\delta_{ij}$ is the Kronecker delta.

We use the Frobenius inner product $\langle U,V\rangle=\tr(U^{\mathsf T}V)$ and the associated norm for real matrices, and write $[U,V]=UV-VU$ for the commutator.
For symmetric matrices, the inner product reduces to $\langle U,V\rangle=\tr(UV)$.
The Simons identity for a minimal immersion \cite{GLZFlat2026,Simons1968} reads
\begin{equation}\label{eq:simons-tensor}
    \Delta A^\alpha=nA^\alpha-\sum_{\beta=1}^q\langle A^\alpha,A^\beta\rangle A^\beta
        -\sum_{\beta=1}^q[A^\beta,[A^\beta,A^\alpha]].
\end{equation}
Taking the inner product with $A^\alpha$ and summing over $\alpha$ gives
\begin{equation}\label{eq:simons}
    \frac{1}{2}\Delta S=|\nabla A|^2+nS
        -\sum_{\alpha,\beta=1}^q\langle A^\alpha,A^\beta\rangle^2
        -\sum_{\alpha,\beta=1}^q|[A^\alpha,A^\beta]|^2.
\end{equation}

The Gram matrix $\G=(\G_{\alpha\beta})$ of the shape operators is defined by
\begin{equation}\label{eq:gram}
 \G_{\alpha\beta}
 =\sum_{i,j=1}^nA^\alpha_{ij}A^\beta_{ij}
 =\langle A^\alpha,A^\beta\rangle.
\end{equation}
Since $\G$ is symmetric and positive semidefinite, its eigenvalues $\lambda_1\ge\cdots\ge\lambda_q\ge0$ satisfy
\begin{equation}\label{eq:gram-eigenvalues}
 \sum_{\alpha=1}^q\lambda_\alpha=\tr\G=S.
\end{equation}
For $q\ge2$, $\rank\G\le1$ implies $\lambda_2=0$. For $1\le k\le q$, let $\sigma_k$ be the elementary symmetric polynomial
\[
 \sigma_k=\sum_{1\le\alpha_1<\cdots<\alpha_k\le q}
              \lambda_{\alpha_1}\cdots\lambda_{\alpha_k},
\]
with $\sigma_0=1$ and $\sigma_k=0$ for $k>q$. The power sums are $\tr(\G^j)=\sum_{\alpha=1}^q\lambda_\alpha^j$ for $j\ge1$. By Newton's identities, we have
\begin{equation}\label{eq:gram-newton}
 k\sigma_k=\sum_{j=1}^k(-1)^{j-1}\sigma_{k-j}\tr(\G^j),
\end{equation}
for every integer $k\ge1$.

At $p\in M$, the second fundamental form is the symmetric bilinear map $A_p:T_pM\times T_pM\to N_pM$ given by
\[
 A_p(e_i(p),e_j(p))=\sum_{\alpha=1}^q A^\alpha_{ij}(p)\nu_\alpha(p).
\]
The first normal space at $p$ is the span of its values \cite{Allendoerfer1939,Erbacher1971}:
\[
 N_1(p):=\Span\{A_p(u,v):u,v\in T_pM\}\subset N_pM.
\]
Its dimension satisfies
\[
 \dim N_1(p)=\dim\Span\left\{A^1_p,\ldots,A^q_p\right\}=\rank\G(p).
\]

Define
\begin{equation}\label{eq:defect}
 E:=S^2-\tr(\G^2)
  =2\sum_{1\le\alpha<\beta\le q}
       \left(|A^\alpha|^2|A^\beta|^2
             -\langle A^\alpha,A^\beta\rangle^2\right)\ge0.
\end{equation}
If $S(p)>0$, then $E(p)=0$ is equivalent to $\dim N_1(p)=1$.

Fix $p\in M$ and choose an orthonormal normal frame that diagonalizes $\G(p)$, so that
\begin{equation}\label{eq:normal-orthogonality}
 \G_{\alpha\beta}=\langle A^\alpha,A^\beta\rangle
 =\delta_{\alpha\beta}|A^\alpha|^2,
\end{equation}
for all $1\le\alpha,\beta\le q$. In this frame, $E=2\sum_{1\le\alpha<\beta\le q}|A^\alpha|^2|A^\beta|^2$.

By \eqref{eq:ricci}, flatness of the normal bundle implies $A^\alpha A^\beta=A^\beta A^\alpha$ for all $1\le\alpha,\beta\le q$.
Under this assumption, \eqref{eq:simons} reduces to
\begin{equation}\label{eq:simons-flat}
    \frac{1}{2}\Delta S=|\nabla A|^2+nS-\tr(\G^2).
\end{equation}
We may also choose an orthonormal tangent frame that diagonalizes all the shape operators simultaneously.
The nonzero shape operators are pairwise orthogonal by \eqref{eq:normal-orthogonality} and hence linearly independent.
Since the immersion is minimal, their diagonal entries satisfy $\sum_{i=1}^n A^\alpha_{ii}=0$.
This is a single nontrivial linear constraint on the $n$ diagonal entries, so the space of trace-free diagonal matrices has dimension $n-1$.
Thus
\begin{equation}\label{eq:gram-rank-bound}
 \rank\G\le\min\{q,n-1\}.
\end{equation}
Under this flatness assumption, at a point $p$ with $E(p)>0$, set $r=\dim N_1(p)=\rank\G(p)$, so that $2\le r\le\min\{q,n-1\}$. Relabel the normal frame so that $A^\alpha(p)\ne0$ for $1\le\alpha\le r$ and $A^\alpha(p)=0$ for $r<\alpha\le q$. In calculations involving only the nonzero shape operators, normal indices range from $1$ to $r$.

\begin{lemma}[Ge--Li--Zhang \cite{GLZFlat2026}]
\label{lem:symmetry}
Let $x:M^n\to\Sph^{n+q}$ be a minimal immersion with flat normal bundle. Then
\[
 \sum_{i,j=1}^nA^\alpha_{ij}A^\beta_{ij;k}
 =\sum_{i,j=1}^nA^\beta_{ij}A^\alpha_{ij;k}.
\]
\end{lemma}

Ge--Li--Zhang \cite[Lemma~5.3]{GLZFlat2026} established the differential identity for $E$ when $S$ is constant. For general $S$, we obtain the following formula with an additional term $\frac{1}{2}|\nabla S|^2$.

\begin{lemma}
\label{lem:evolution}
Let $x:M^n\to\Sph^{n+q}$ be a minimal immersion with flat normal bundle.
Fix $p\in M$ and choose an orthonormal normal frame that diagonalizes $\G(p)$. Then
\begin{equation}\label{eq:E-evolution}
 \begin{aligned}
 \frac{1}{4}\Delta E
 =&\sum_{\alpha=1}^q\left(S-|A^\alpha|^2\right)|\nabla A^\alpha|^2
       +\frac{1}{2}|\nabla S|^2-2\sum_{\alpha,\beta=1}^q\sum_{k=1}^n
       \left(\sum_{i,j=1}^nA^\alpha_{ij}A^\beta_{ij;k}\right)^2\\
 &+\frac{2n-S}{2}E
       +3\sum_{1\le\alpha<\beta<\gamma\le q}|A^\alpha|^2|A^\beta|^2|A^\gamma|^2.
 \end{aligned}
\end{equation}
\end{lemma}
\begin{proof}
By \eqref{eq:simons-flat}, we have
\[
 \frac{1}{4}\Delta(S^2)
 =S|\nabla A|^2+nS^2-S\tr(\G^2)+\frac{1}{2}|\nabla S|^2.
\]
Under the flatness assumption, \eqref{eq:simons-tensor} reduces to
\[
    \Delta A^\alpha=nA^\alpha-\sum_{\beta=1}^q\G_{\alpha\beta}A^\beta.
\]
Differentiating $\G$ and using Lemma~\ref{lem:symmetry} and this identity, we obtain
\begin{align*}
 (\nabla_k\G)_{\alpha\beta}
 &=2\sum_{i,j=1}^nA^\alpha_{ij}A^\beta_{ij;k},\\
 (\Delta\G)_{\alpha\beta}
 &=2n\G_{\alpha\beta}-2(\G^2)_{\alpha\beta}
       +2\sum_{i,j,k=1}^nA^\alpha_{ij;k}A^\beta_{ij;k}.
\end{align*}
Since $\tr(\G^2)=|\G|^2$, the tensor product rule gives
\[
 \frac{1}{2}\Delta\tr(\G^2)
 =\langle\Delta\G,\G\rangle+|\nabla\G|^2.
\]
Substituting the preceding formulas and using \eqref{eq:normal-orthogonality}, we obtain
\begin{align*}
 \frac{1}{4}\Delta\tr(\G^2)=n\tr(\G^2)-\tr(\G^3)+\sum_{\alpha=1}^q|A^\alpha|^2|\nabla A^\alpha|^2+2\sum_{\alpha,\beta=1}^q\sum_{k=1}^n\left(\sum_{i,j=1}^nA^\alpha_{ij}A^\beta_{ij;k}\right)^2.
\end{align*}

Subtracting the formulas for $\frac{1}{4}\Delta(S^2)$ and $\frac{1}{4}\Delta\tr(\G^2)$, we obtain
\[
\begin{aligned}
 \frac{1}{4}\Delta E
 &=\sum_{\alpha=1}^q(S-|A^\alpha|^2)|\nabla A^\alpha|^2
       +\frac{1}{2}|\nabla S|^2
       -2\sum_{\alpha,\beta=1}^q\sum_{k=1}^n\left(\sum_{i,j=1}^nA^\alpha_{ij}A^\beta_{ij;k}\right)^2\\
 &\quad+nS^2-(S+n)\tr(\G^2)+\tr(\G^3).
\end{aligned}
\]
Substituting $\sigma_1=S$ and $2\sigma_2=E$ into \eqref{eq:gram-newton} with $k=3$, we obtain
\[
\begin{aligned}
 \tr(\G^3)=S\tr(\G^2)-\frac{1}{2}SE+3\sigma_3=\frac{3}{2}S\tr(\G^2)-\frac{1}{2}S^3
       +3\sum_{1\le\alpha<\beta<\gamma\le q}|A^\alpha|^2|A^\beta|^2|A^\gamma|^2.
\end{aligned}
\]
Hence
\[
 \begin{aligned}
 nS^2-(S+n)\tr(\G^2)+\tr(\G^3)=\frac{2n-S}{2}E
       +3\sum_{1\le\alpha<\beta<\gamma\le q}|A^\alpha|^2|A^\beta|^2|A^\gamma|^2.
 \end{aligned}
\]
Substituting the above identity into the formula for $\frac{1}{4}\Delta E$, we obtain \eqref{eq:E-evolution}.
\end{proof}

\section{Weighted matrix inequalities}\label{sec:matrix}

In this section, we prove two weighted matrix inequalities. We first establish an inequality for real symmetric matrices with zero weighted trace.

Let $r\ge2$ and $a_1,\ldots,a_r>0$ satisfy $\sum_{i=1}^r a_i=1$. Set
\[
 \Lambda=\diag(a_1,\ldots,a_r),\qquad D=\Lambda(I_r-\Lambda),
\]
where $D$ is positive definite.

\begin{lemma}\label{lem:matrix}
For every real symmetric $r\times r$ matrix $T$, define
\[
    Q_\Lambda(T)
    =3\tr(T^2D)-2\tr(\Lambda T\Lambda T)
        +2\left(\tr(\Lambda T)\right)^2.
\]
If $\tr(DT)=0$, then
\begin{equation}\label{eq:matrix}
 2TDT\preceq Q_\Lambda(T)I_r,
\end{equation}
where $\preceq$ denotes the positive semidefinite order.
\end{lemma}

\begin{proof}
It suffices to prove the following inequality for every unit vector $v\in\R^r$ and every $w\in\R^r$
\[
 2(w^{\mathsf T}\sqrt D\,Tv)^2\le Q_\Lambda(T)|w|^2.
\]

Separating the diagonal and off diagonal terms in $Q_\Lambda(T)$, we obtain
\[
\begin{aligned}
Q_\Lambda(T)=&3\left(\sum_{i=1}^r D_{ii}T_{ii}^2
         +\sum_{1\le i<j\le r}(D_{ii}+D_{jj})T_{ij}^2\right)-2\left(\sum_{i=1}^r a_i^2T_{ii}^2
              +2\sum_{1\le i<j\le r}a_i a_jT_{ij}^2\right)\\
&+2\left(\sum_{i=1}^r a_i^2T_{ii}^2
              +2\sum_{1\le i<j\le r}a_i a_jT_{ii}T_{jj}\right)\\
=&3\sum_{i=1}^rD_{ii}T_{ii}^2
  +4\sum_{1\le i<j\le r}a_i a_jT_{ii}T_{jj}
  +2\sum_{1\le i<j\le r}c_{ij}T_{ij}^2,
\end{aligned}
\]
where $c_{ij}=\frac{3}{2}(D_{ii}+D_{jj})-2a_i a_j$ for $i\ne j$. Since $a_i+a_j\le1$, we have $c_{ij}=a_i a_j+\frac{3}{2}(a_i+a_j)(1-a_i-a_j)>0$.

In the normalized coordinates $(\sqrt{D_{ii}}T_{ii})_{i=1}^r$, the diagonal part of $Q_\Lambda$ has coefficient matrix $3I_r+2P$, where $P_{ii}=0$ and $P_{ij}=\frac{a_i a_j}{\sqrt{D_{ii}D_{jj}}}$ for $i\ne j$. Using $D_{ii}=a_i(1-a_i)$ and $\sum_{\substack{j=1\\j\ne i}}^ra_j=1-a_i$, we obtain for every $y\in\R^r$
\[
 \begin{aligned}
 y^{\mathsf T}(I_r+P)y
 &=\sum_{i=1}^r y_i^2+2\sum_{1\le i<j\le r}\frac{a_i a_j}{\sqrt{D_{ii}D_{jj}}}y_i y_j\\
 &=\sum_{i=1}^r\frac{a_i\sum_{\substack{j=1\\j\ne i}}^ra_j}{D_{ii}}y_i^2
   +2\sum_{1\le i<j\le r}\frac{a_i a_j}{\sqrt{D_{ii}D_{jj}}}y_i y_j\\
 &=\sum_{1\le i<j\le r}a_i a_j
   \left(\frac{y_i}{\sqrt{D_{ii}}}
             +\frac{y_j}{\sqrt{D_{jj}}}\right)^2\ge0.
 \end{aligned}
\]
Hence all eigenvalues of $P$ are at least $-1$, so that $3I_r+2P\succeq I_r$. The decomposition of $Q_\Lambda$ then gives
\[
 Q_\Lambda(T)\ge
 \sum_{i=1}^r D_{ii}T_{ii}^2
 +2\sum_{1\le i<j\le r}c_{ij}T_{ij}^2.
\]
Since $D_{ii}>0$ and $c_{ij}>0$, the right-hand side is strictly positive for every nonzero symmetric matrix $T$. Thus $Q_\Lambda$ is positive definite.

By direct computation, we have
\[
 \sqrt2\,w^{\mathsf T}\sqrt D\,Tv
 =\sqrt2\sum_{i=1}^r w_i v_i\left(\sqrt{D_{ii}}T_{ii}\right)
   +\sqrt2\sum_{1\le i<j\le r}
       (\sqrt{D_{ii}}w_i v_j+\sqrt{D_{jj}}w_j v_i)T_{ij}.
\]
The weighted trace gives $\tr(DT)=\sum_{i=1}^rD_{ii}T_{ii}=\left\langle(\sqrt{D_{ii}}T_{ii})_{i=1}^r,(\sqrt{D_{ii}})_{i=1}^r\right\rangle$.
This trace constraint involves only the diagonal entries of $T$, leaving its off diagonal entries unrestricted.
The condition $\tr(DT)=0$ therefore means that $(\sqrt{D_{ii}}T_{ii})_{i=1}^r$ is orthogonal to $(\sqrt{D_{ii}})_{i=1}^r$.
Using the above constraint, we obtain
\[
 \sum_{i=1}^r\sqrt{D_{ii}}\,w_i v_i T_{ii}
 =\sum_{i=1}^r\left(w_i v_i-\frac{\sqrt{D_{ii}}}{\tr D}
              \sum_{k=1}^r\sqrt{D_{kk}}\,w_k v_k\right)\sqrt{D_{ii}}T_{ii}.
\]
Moreover, for each $i$
\[
 \sum_{\substack{j=1\\j\ne i}}^r\frac{a_i a_j}{\sqrt{D_{ii}D_{jj}}}\sqrt{D_{jj}}
 =\frac{a_i}{\sqrt{D_{ii}}}\sum_{\substack{j=1\\j\ne i}}^ra_j
 =\frac{a_i(1-a_i)}{\sqrt{D_{ii}}}
 =\sqrt{D_{ii}}.
\]
Thus $(\sqrt{D_{ii}})_{i=1}^r$ is an eigenvector of $P$ with eigenvalue $1$, so that
\[
 \sum_{j=1}^r\left((3I_r+2P)^{-1}\right)_{ij}\sqrt{D_{jj}}
 =\frac{\sqrt{D_{ii}}}{5}.
\]
For $(T_{11},\ldots,T_{rr})\ne0$, the above identities and the Cauchy--Schwarz inequality for $3I_r+2P$ give
\[
    \begin{aligned}
        &\frac{2\left(\sum_{i=1}^r\sqrt{D_{ii}}\,w_i v_i T_{ii}\right)^2}
            {3\sum_{i=1}^rD_{ii}T_{ii}^2+4\sum_{1\le i<j\le r}a_i a_jT_{ii}T_{jj}}\\
        \le{}&2\sum_{i,j=1}^r
            \left(w_i v_i-\frac{\sqrt{D_{ii}}}{\tr D}
                \sum_{k=1}^r\sqrt{D_{kk}}\,w_k v_k\right)
            \left((3I_r+2P)^{-1}\right)_{ij}\times
            \left(w_j v_j-\frac{\sqrt{D_{jj}}}{\tr D}
                \sum_{k=1}^r\sqrt{D_{kk}}\,w_k v_k\right)\\
        ={}&2\sum_{i,j=1}^rw_i v_i
            \left[\left((3I_r+2P)^{-1}\right)_{ij}
                -\frac{\sqrt{D_{ii}D_{jj}}}{5\tr D}\right]w_j v_j,
    \end{aligned}
\]
with equality for $T_{ii}=\frac{1}{\sqrt{D_{ii}}}\sum_{j=1}^r\left[\left((3I_r+2P)^{-1}\right)_{ij}-\frac{\sqrt{D_{ii}D_{jj}}}{5\tr D}\right]w_j v_j$, provided these entries are not all zero.
Hence
\begin{align*}
 &\sup_{\substack{(T_{11},\ldots,T_{rr})\in\R^r\setminus\{0\}\\
                   \sum_{i=1}^rD_{ii}T_{ii}=0}}
   \frac{2\left(\sum_{i=1}^r\sqrt{D_{ii}}\,w_i v_i T_{ii}\right)^2}
        {3\sum_{i=1}^rD_{ii}T_{ii}^2+4\sum_{1\le i<j\le r}a_i a_jT_{ii}T_{jj}}\\
 =&2\sum_{i,j=1}^rw_i v_i
   \left[\left((3I_r+2P)^{-1}\right)_{ij}
   -\frac{\sqrt{D_{ii}D_{jj}}}{5\tr D}\right]w_j v_j.
\end{align*}
For each $i<j$, the linear coefficient of $T_{ij}$ is $\sqrt2(\sqrt{D_{ii}}w_i v_j+\sqrt{D_{jj}}w_j v_i)$, and its quadratic coefficient in $Q_\Lambda$ is $2c_{ij}$. Since there are no cross terms between the diagonal and off diagonal variables, we obtain
\[
\begin{aligned}
 \sup_{\substack{T=T^{\mathsf T}\ne0\\\tr(DT)=0}}
        \frac{2(w^{\mathsf T}\sqrt D\,Tv)^2}{Q_{\Lambda}(T)}=&2\sum_{i,j=1}^rw_i v_i
       \left[\left((3I_r+2P)^{-1}\right)_{ij}
       -\frac{\sqrt{D_{ii}D_{jj}}}{5\tr D}\right]w_j v_j\\
 &+\sum_{1\le i<j\le r}
       \frac{\left(\sqrt{D_{ii}}w_i v_j+\sqrt{D_{jj}}w_j v_i\right)^2}{c_{ij}}.
\end{aligned}
\]

Using $\frac{a_i^2}{D_{ii}}=\frac{a_i}{1-a_i}>0$, we obtain
\[
 \begin{aligned}
 P+\diag\left(\frac{a_1}{1-a_1},\ldots,\frac{a_r}{1-a_r}\right)=\left(\frac{a_i a_j}{\sqrt{D_{ii}D_{jj}}}\right)_{i,j=1}^r.
 \end{aligned}
\]
Thus the matrix on the right has rank one, and its kernel has dimension $r-1$. For every nonzero vector $y$ in this kernel, the preceding identity gives
\[
 y^{\mathsf T}Py
 =-\sum_{i=1}^r\frac{a_i}{1-a_i}y_i^2<0.
\]
Since $P$ is negative definite on this $(r-1)$-dimensional kernel, it has at most one positive eigenvalue. Since $(\sqrt{D_{ii}})_{i=1}^r$ is an eigenvector with eigenvalue $1$, symmetry and $P\succeq-I_r$ imply that all eigenvalues on its orthogonal complement lie in $[-1,0]$.

By convexity, for $-1\le\lambda\le0$,
\[
 \frac{1}{3+2\lambda}
 \le -\lambda+\frac{1+\lambda}{3}
 =\frac{1-2\lambda}{3}.
\]
Hence $(3I_r+2P)^{-1}\preceq\frac{1}{3}(I_r-2P)$ on this orthogonal complement. On the line spanned by $(\sqrt{D_{ii}})_{i=1}^r$, the two operators have eigenvalues $\frac{1}{5}$ and $-\frac{1}{3}$, respectively. The orthogonal projection onto this line is $\left(\sqrt{D_{ii}D_{jj}}\right)_{i,j=1}^r/\tr D$. Subtracting $\frac{1}{5}$ of this projection from the left-hand side and adding $\frac{1}{3}$ of it to the right-hand side extends both restricted operators by zero along the line. Therefore,
\begin{equation*}\label{eq:restricted-spectral-bound}
 (3I_r+2P)^{-1}
 -\frac{\left(\sqrt{D_{ii}D_{jj}}\right)_{i,j=1}^r}{5\tr D}
 \preceq\frac{1}{3}\left[
 I_r-2P+\frac{\left(\sqrt{D_{ii}D_{jj}}\right)_{i,j=1}^r}{\tr D}
 \right].
\end{equation*}

Writing the preceding matrix estimate in components, using the zero diagonal of $P$ and its off diagonal entries $\frac{a_i a_j}{\sqrt{D_{ii}D_{jj}}}$, we obtain
\begin{equation}\label{eq:matrix-dual-upper}
    \begin{aligned}
        \sup_{\substack{T=T^{\mathsf T}\ne0\\\tr(DT)=0}}
            \frac{2(w^{\mathsf T}\sqrt D\,Tv)^2}{Q_{\Lambda}(T)}\le&\frac{2}{3}\left\{
            \sum_{i=1}^r(w_i v_i)^2
            -4\sum_{1\le i<j\le r}\frac{a_i a_j}{\sqrt{D_{ii}D_{jj}}}w_i v_i w_j v_j\right.\left.+\frac{\left(\sum_{i=1}^r\sqrt{D_{ii}}\,w_i v_i\right)^2}{\tr D}
            \right\}\\
        &+\sum_{1\le i<j\le r}
            \frac{(\sqrt{D_{ii}}w_i v_j+\sqrt{D_{jj}}w_j v_i)^2}{c_{ij}}.
    \end{aligned}
\end{equation}

It remains to bound the right-hand side of \eqref{eq:matrix-dual-upper} by $|w|^2$. Since
\begin{equation}\label{eq:diagonal-weight-bound}
 \tr D-2D_{ii}
 =2\sum_{\substack{1\le j<k\le r\\j,k\ne i}}a_j a_k\ge0,
\end{equation}
the coefficient $\frac{2}{3}(1+\frac{D_{ii}}{\tr D})$ of $(w_i v_i)^2$ on the right-hand side of \eqref{eq:matrix-dual-upper} is at most $1$.
The remaining terms group into pairs because $(w_i v_i)(w_j v_j)=(w_i v_j)(w_j v_i)$.
Expanding the right-hand side of \eqref{eq:matrix-dual-upper} and using the bound on the diagonal coefficients, we obtain
\[
\begin{aligned}
 \sup_{\substack{T=T^{\mathsf T}\ne0\\\tr(DT)=0}}
 \frac{2(w^{\mathsf T}\sqrt D\,Tv)^2}{Q_\Lambda(T)}
 \le
 &\sum_{1\le i<j\le r}\left[
 \frac{D_{ii}}{c_{ij}}(w_i v_j)^2
 +\frac{D_{jj}}{c_{ij}}(w_j v_i)^2
 +\frac{2\sqrt{D_{ii}D_{jj}}}{c_{ij}}b_{ij}w_i v_j w_j v_i\right]\\
 &+\sum_{i=1}^r(w_i v_i)^2,
\end{aligned}
\]
where, for each $i<j$,
\[
    b_{ij}=1+\frac{2c_{ij}}{3\tr D}
        -\frac{4c_{ij}a_i a_j}{3D_{ii}D_{jj}}.
\]
To bound each summand over $i<j$ by $(w_i v_j)^2+(w_j v_i)^2$, it suffices to show that
\[
 \begin{pmatrix}
 c_{ij}-D_{ii} & -\sqrt{D_{ii}D_{jj}}\,b_{ij}\\
 -\sqrt{D_{ii}D_{jj}}\,b_{ij} & c_{ij}-D_{jj}
 \end{pmatrix}\succeq0.
\]
Since $a_i+a_j\le1$,
\begin{equation}\label{eq:pair-coefficient-positive}
 c_{ij}-D_{ii}=\frac{1}{2}(a_i+3a_j)(1-a_i-a_j)\ge0,
\end{equation}
and symmetry in $i,j$ gives $c_{ij}\ge\max\{D_{ii},D_{jj}\}>0$. The diagonal entries are therefore nonnegative, so positive semidefiniteness is equivalent to a nonnegative determinant, or
\begin{equation}\label{eq:pair-spectral-bound}
 |b_{ij}|\le
 \sqrt{\frac{(c_{ij}-D_{ii})(c_{ij}-D_{jj})}{D_{ii}D_{jj}}}.
\end{equation}
For fixed $i<j$, direct calculation gives
\[
 \sqrt{\frac{(c_{ij}-D_{ii})(c_{ij}-D_{jj})}{D_{ii}D_{jj}}}
 =\frac{(1-a_i-a_j)\sqrt{3(a_i+a_j)^2+4a_i a_j}}
        {2\sqrt{a_i a_j(1-a_i)(1-a_j)}}.
\]
We next derive upper and lower bounds for $b_{ij}$.
From the definition of $D$, we have
\begin{equation}\label{eq:trace-D-expansion}
\begin{aligned}
 \tr(D)=1-\sum_{k=1}^r a_k^2
 =1-(a_i+a_j)^2+2a_i a_j-\sum_{\substack{k=1\\k\ne i,j}}^ra_k^2.
\end{aligned}
\end{equation}
Substituting $0\le\sum_{\substack{k=1\\k\ne i,j}}^ra_k^2\le(1-a_i-a_j)^2$ into \eqref{eq:trace-D-expansion} gives
\[
 \frac{2}{3}(2c_{ij}+a_i a_j)\le\tr D
 \le D_{ii}+D_{jj}+1-a_i-a_j.
\]
\begin{samepage}
\noindent Using $D_{ii}D_{jj}=a_i a_j(1-a_i)(1-a_j)$, the lower bound for $\tr D$ gives an upper bound for $b_{ij}$:
\[
 \begin{aligned}
 b_{ij}
 &\le1-\frac{4c_{ij}}{3(1-a_i)(1-a_j)}
       +\frac{c_{ij}}{2c_{ij}+a_i a_j}\\
 &=\frac{3(1-a_i-a_j)^2}{2(1-a_i)(1-a_j)}
   -\frac{(1-a_i-a_j)\left[2(a_i+a_j)c_{ij}+a_i a_j\right]}
          {2(2c_{ij}+a_i a_j)(1-a_i)(1-a_j)}\\
 &\le\frac{3(1-a_i-a_j)^2}{2(1-a_i)(1-a_j)}.
 \end{aligned}
\]
Since
\[
    (1-a_i-a_j)^2\le(1-a_i)(1-a_j),\qquad
    9a_i a_j\le3(a_i+a_j)^2+4a_i a_j,
\]
this upper bound is at most the right-hand side of \eqref{eq:pair-spectral-bound}.
\end{samepage}

\begin{samepage}
\noindent Similarly, the upper bound for $\tr D$ gives a lower bound for $b_{ij}$:
\[
 \begin{aligned}
 b_{ij}
 &\ge1-\frac{4c_{ij}}{3(1-a_i)(1-a_j)}
       +\frac{2c_{ij}}{3(D_{ii}+D_{jj}+1-a_i-a_j)}\\
 &=-\frac{(a_i+a_j)(1-a_i-a_j)}{2(1-a_i)(1-a_j)}
   +\frac{(1-a_i-a_j)^2}{(1-a_i)(1-a_j)}
      \left(1+\frac{c_{ij}}{3(D_{ii}+D_{jj}+1-a_i-a_j)}\right)\\
 &\ge-\frac{(a_i+a_j)(1-a_i-a_j)}{2(1-a_i)(1-a_j)}.
 \end{aligned}
\]
Since
\[
    a_i a_j\le(1-a_i)(1-a_j),\qquad
    (a_i+a_j)^2\le3(a_i+a_j)^2+4a_i a_j,
\]
the absolute value of this lower bound is at most the right-hand side of \eqref{eq:pair-spectral-bound}.
\end{samepage}
Together, these estimates prove \eqref{eq:pair-spectral-bound}, including when $a_i+a_j=1$.
Hence
\[
\begin{aligned}
 \sup_{\substack{T=T^{\mathsf T}\ne0\\\tr(DT)=0}}
 \frac{2(w^{\mathsf T}\sqrt D\,Tv)^2}{Q_\Lambda(T)}\le\sum_{i=1}^r(w_i v_i)^2
   +\sum_{1\le i<j\le r}\left[(w_i v_j)^2+(w_j v_i)^2\right]=\sum_{i,j=1}^rw_i^2v_j^2=|w|^2.
\end{aligned}
\]
This proves \eqref{eq:matrix}.
\end{proof}

Using Lemma~\ref{lem:matrix}, we next prove the following matrix inequality with a parameter.

For $0\le p\le\frac{2}{3}$, set
\[
    D_p=D-\frac{p\tr(D)}{2}\Lambda.
\]

\begin{lemma}\label{lem:matrix-parameter}
For every real symmetric $r\times r$ matrix $\widetilde{T}$, define
\begin{equation}\label{eq:lu-q-p}
    \begin{aligned}
        Q_{\Lambda,p}(\widetilde{T})
        =3\tr(\widetilde{T}^2D_p)-2\tr(\Lambda \widetilde{T}\Lambda \widetilde{T})+\left[2+p(1-p)\tr(D)\right]\left(\tr(\Lambda \widetilde{T})\right)^2.
    \end{aligned}
\end{equation}
If $\tr(D_p\widetilde{T})=0$, then
\begin{equation}\label{eq:lu-matrix}
 2\left(1-\frac{3p}{2}\right)\widetilde{T}D_p\widetilde{T}
 \preceq Q_{\Lambda,p}(\widetilde{T})I_r.
\end{equation}
\end{lemma}

The next two lemmas estimate the remainder $\mathcal R_\Lambda(T)$ and the mixed term $\tr(\Lambda T)(TD+DT)$ in the proof of Lemma~\ref{lem:matrix-parameter}. For every real symmetric $r\times r$ matrix $T$, define
\[
\begin{aligned}
 \mathcal R_\Lambda(T)
 &=Q_\Lambda(T)-\tr(D)\tr(\Lambda T^2)\\
 &=3\tr(T^2D)-2\tr(\Lambda T\Lambda T)
       +2\left(\tr(\Lambda T)\right)^2-\tr(D)\tr(\Lambda T^2).
\end{aligned}
\]
\begin{lemma}\label{lem:lu-auxiliary}
If $T$ is a real symmetric $r\times r$ matrix satisfying $\tr(DT)=0$, then
\begin{equation}\label{eq:lu-diagonal-remainder}
\mathcal R_\Lambda(T)\ge2\sum_{i=1}^r a_i\left((2a_i-1)T_{ii}-\tr(\Lambda T)\right)^2\ge0.
\end{equation}
\end{lemma}
\begin{proof}
Since $T$ is symmetric,
\[
 \tr(\Lambda T^2)=\sum_{i=1}^r a_iT_{ii}^2+\sum_{1\le i<j\le r}(a_i+a_j)T_{ij}^2.
\]
Combining this with the expansion of $Q_\Lambda(T)$ and using $D_{ii}=a_i(1-a_i)$, we obtain
\begin{equation}\label{eq:lu-remainder-decomposition}
\begin{aligned}
 \mathcal R_\Lambda(T)
 &=\sum_{i=1}^r\left[3D_{ii}-2a_i^2-\tr(D)a_i\right]T_{ii}^2
       +2\left(\tr(\Lambda T)\right)^2\\
 &+\sum_{1\le i<j\le r}\left[3(D_{ii}+D_{jj})-4a_i a_j-\tr(D)(a_i+a_j)\right]T_{ij}^2\\
 &=\mathcal R_{\Lambda,\mathrm{diag}}(T)+\sum_{1\le i<j\le r}h_{ij}T_{ij}^2,
\end{aligned}
\end{equation}
where the diagonal part is
\[
 \mathcal R_{\Lambda,\mathrm{diag}}(T)=\sum_{i=1}^r a_i(3-\tr(D)-5a_i)T_{ii}^2+2\left(\tr(\Lambda T)\right)^2.
\]
From \eqref{eq:trace-D-expansion}, we obtain
\begin{equation}\label{eq:lu-offdiagonal-coefficient}
\begin{aligned}
 h_{ij}
 &=3(D_{ii}+D_{jj})-4a_i a_j-\tr(D)(a_i+a_j)\\
 &=(a_i+a_j)(1-a_i-a_j)(2-a_i-a_j)+2a_i a_j(1-a_i-a_j)+(a_i+a_j)\sum_{\substack{k=1\\k\ne i,j}}^ra_k^2.
\end{aligned}
\end{equation}
Since $1-a_i-a_j=\sum_{\substack{k=1\\k\ne i,j}}^ra_k\ge0$ and $2-a_i-a_j\ge1$, we have $h_{ij}\ge0$. The trace constraint gives
\[
 0=\tr(DT)=\sum_{i=1}^r a_i(1-a_i)T_{ii}
   =\sum_{i=1}^r a_iT_{ii}-\sum_{i=1}^r a_i^2T_{ii},
\]
and hence
\[
 \sum_{i=1}^r a_i^2T_{ii}=\sum_{i=1}^r a_iT_{ii}=\tr(\Lambda T).
\]
By the above identity, we obtain
\begin{equation}\label{eq:lu-strong-diagonal-form}
\begin{aligned}
 \mathcal R_{\Lambda,\mathrm{diag}}(T)
       -2\sum_{i=1}^r a_i\left((2a_i-1)T_{ii}-\tr(\Lambda T)\right)^2=\sum_{i=1}^r a_i\psi(a_i)T_{ii}^2+4\left(\tr(\Lambda T)\right)^2,
\end{aligned}
\end{equation}
where $\psi(x)=1-\tr(D)+3x-8x^2$.

Since $h_{ij}\ge0$, it suffices to show that the right-hand side of \eqref{eq:lu-strong-diagonal-form} is nonnegative. When $r=2$, we have $a_1+a_2=1$ and $D=a_1a_2I_2$, so the trace constraint becomes
\[
 0=\tr(DT)=a_1a_2(T_{11}+T_{22}),
\]
and hence $T_{22}=-T_{11}$ and $\tr(\Lambda T)=(a_1-a_2)T_{11}$. Substituting these relations and $\tr(D)=2a_1a_2$ into the right-hand side of \eqref{eq:lu-strong-diagonal-form}, we obtain
\[
\begin{aligned}
 &\sum_{i=1}^2a_i\psi(a_i)T_{ii}^2
       +4\left(\tr(\Lambda T)\right)^2\\
 =&\left[1-2a_1a_2+3(a_1^2+a_2^2)-8(a_1^3+a_2^3)+4(a_1-a_2)^2\right]T_{11}^2\\
 =&\left[1-2a_1a_2+3(1-2a_1a_2)-8(1-3a_1a_2)+4(1-4a_1a_2)\right]T_{11}^2=0.
\end{aligned}
\]
This proves \eqref{eq:lu-diagonal-remainder} for $r=2$.

We may now assume $r\ge3$ and, after relabeling, $a_1\ge\cdots\ge a_r$. For $i\ge2$, the ordering and \eqref{eq:trace-D-expansion} give $a_i<\frac{1}{2}$ and $1-\tr(D)\ge2a_i^2$, so
\[
 \psi(a_i)\ge3a_i(1-2a_i)>0.
\]
We may therefore assume $\psi(a_1)<0$. It remains to bound $-a_1\psi(a_1)T_{11}^2$ by the remaining terms on the right-hand side of \eqref{eq:lu-strong-diagonal-form}.

Set $\tau=\sqrt{\sum_{j=2}^r a_j^2}$. Since $r\ge3$ and all $a_j$ are positive, $0<a_i<\tau<1-a_1$ for $i\ge2$. If $a_1\le\tau$, then $a_1<\frac{1}{2}$ and $\psi(a_1)\ge3a_1(1-2a_1)>0$, contrary to $\psi(a_1)<0$. Hence $0<\tau<\min\{a_1,1-a_1\}\le\frac{1}{2}$, and
\[
 \psi(\tau)=a_1^2+3\tau-7\tau^2
 \ge3\tau(1-2\tau)>0.
\]

Since $0<a_i<\tau$ for $i\ge2$, the concavity of $\psi$ gives
\[
\begin{aligned}
 \psi(a_i)\ge\left(1-\frac{a_i}{\tau}\right)\psi(0)+\frac{a_i}{\tau}\psi(\tau)=\left(1-\frac{a_i}{\tau}\right)(a_1^2+\tau^2)+\frac{a_i}{\tau}\psi(\tau).
\end{aligned}
\]
The weights $a_i(1-\frac{a_i}{\tau})$ and $a_i^2$ are positive for $i\ge2$, and their sums are $1-a_1-\tau$ and $\tau^2$, respectively. By the Cauchy--Schwarz inequality, we obtain
\[
\begin{aligned}
 \left(\sum_{i=2}^r a_i\left(1-\frac{a_i}{\tau}\right)T_{ii}\right)^2
 \le&(1-a_1-\tau)\sum_{i=2}^r a_i\left(1-\frac{a_i}{\tau}\right)T_{ii}^2,\\
 \left(\sum_{i=2}^r a_i^2T_{ii}\right)^2
 \le&\tau^2\sum_{i=2}^r a_i^2T_{ii}^2.
\end{aligned}
\]
Applying these inequalities to the lower bound for $\psi(a_i)$, we obtain
\[
\begin{aligned}
 \sum_{i=2}^r a_i \psi(a_i)T_{ii}^2
 &\ge (a_1^2+\tau^2)\sum_{i=2}^r a_i\left(1-\frac{a_i}{\tau}\right)T_{ii}^2
       +\frac{\psi(\tau)}{\tau}\sum_{i=2}^r a_i^2T_{ii}^2\\
 &\ge \frac{a_1^2+\tau^2}{1-a_1-\tau}
       \left(\sum_{i=2}^r a_i\left(1-\frac{a_i}{\tau}\right)T_{ii}\right)^2
       +\frac{\psi(\tau)}{\tau^3}\left(\sum_{i=2}^r a_i^2T_{ii}\right)^2.
\end{aligned}
\]
The trace constraint gives
\[
 \sum_{i=2}^r a_iT_{ii}=\tr(\Lambda T)-a_1T_{11},\qquad
 \sum_{i=2}^r a_i^2T_{ii}=\tr(\Lambda T)-a_1^2T_{11}.
\]
Consequently,
\[
\begin{aligned}
 \sum_{i=2}^r a_i\left(1-\frac{a_i}{\tau}\right)T_{ii}
 &=\tr(\Lambda T)-a_1T_{11}
       -\frac{\tr(\Lambda T)-a_1^2T_{11}}{\tau}\\
 &=-\frac{(1-\tau)\tr(\Lambda T)+a_1(\tau-a_1)T_{11}}{\tau}.
\end{aligned}
\]
Substituting the above identities into the lower bound gives
\begin{equation}\label{eq:lu-tail-energy}
\begin{aligned}
 &\sum_{i=1}^r a_i \psi(a_i)T_{ii}^2+4\left(\tr(\Lambda T)\right)^2\\
 \ge& a_1\psi(a_1)T_{11}^2+4\left(\tr(\Lambda T)\right)^2+\frac{\psi(\tau)}{\tau^3}
       \left(\tr(\Lambda T)-a_1^2T_{11}\right)^2\\
 &+\frac{a_1^2+\tau^2}{(1-a_1-\tau)\tau^2}
       \left[(1-\tau)\tr(\Lambda T)+a_1(\tau-a_1)T_{11}\right]^2\\[4pt]
 =&\left[4+\frac{\psi(\tau)}{\tau^3}
       +\frac{(a_1^2+\tau^2)(1-\tau)^2}{(1-a_1-\tau)\tau^2}\right]
       \left(\tr(\Lambda T)\right)^2+2\bigg[\frac{a_1(a_1^2+\tau^2)(\tau-a_1)(1-\tau)}{(1-a_1-\tau)\tau^2}\\
       &-\frac{a_1^2\psi(\tau)}{\tau^3}\bigg]
       T_{11}\tr(\Lambda T)+\left[a_1\psi(a_1)+\frac{a_1^4\psi(\tau)}{\tau^3}
       +\frac{a_1^2(a_1^2+\tau^2)(\tau-a_1)^2}{(1-a_1-\tau)\tau^2}\right]T_{11}^2.
\end{aligned}
\end{equation}

For fixed $T_{11}$, the lower bound in \eqref{eq:lu-tail-energy} is quadratic in $\tr(\Lambda T)$, with positive leading coefficient since $\psi(\tau)>0$ and $1-a_1-\tau>0$. Substituting $\psi(a_1)=\tau^2+3a_1-7a_1^2$ and $\psi(\tau)=a_1^2+3\tau-7\tau^2$, we compute
\begin{equation}\label{eq:lu-tail-discriminant}
\begin{aligned}
 \operatorname{disc}
 =-\frac{4a_1T_{11}^2}{\tau^3}\Bigl\{&
 (1-a_1-\tau)^2(a_1+\tau)^3(3-a_1-\tau)\\
 &+a_1\tau(1-a_1-\tau)\left[2+4(a_1+\tau)^2+7(1-a_1-\tau)^3\right]\\
 &+4a_1^2\tau^2\left[(a_1+\tau)(3-2a_1-2\tau)-a_1\tau\right]\Bigr\}.
\end{aligned}
\end{equation}
Since
\[
\begin{aligned}
 (a_1+\tau)(3-2a_1-2\tau)-a_1\tau=a_1+\tau(1-a_1)+2(1-a_1-\tau)(a_1+\tau)>0,
\end{aligned}
\]
every term in braces is nonnegative, so the discriminant is nonpositive and the polynomial is nonnegative. By \eqref{eq:lu-tail-energy}, the right-hand side of \eqref{eq:lu-strong-diagonal-form} is therefore nonnegative, proving the required bound.
\end{proof}

\begin{lemma}\label{lem:lu-mixed}
If $T$ is a real symmetric $r\times r$ matrix satisfying $\tr(DT)=0$, then
\begin{equation}\label{eq:lu-mixed}
 \mathcal R_\Lambda(T) I_r+\tr(D)\left(\tr(\Lambda T)\right)^2I_r+\frac{2\tr(D)}{3}T\Lambda T-\tr(\Lambda T)(TD+DT)\succeq0.
\end{equation}
\end{lemma}
\begin{proof}
Suppose first that $\tr(D)\ge\frac{1}{2}$. Completing the square, we obtain
\begin{align*}
 &\frac{2\tr(D)}{3}T\Lambda T-\tr(\Lambda T)(TD+DT)+\tr(D)\left(\tr(\Lambda T)\right)^2I_r\\
 =&\frac{2\tr(D)}{3}
 \left[T-\frac{3\tr(\Lambda T)}{2\tr(D)}(I_r-\Lambda)\right]
 \Lambda
 \left[T-\frac{3\tr(\Lambda T)}{2\tr(D)}(I_r-\Lambda)\right]\\
 &+\left(\tr(\Lambda T)\right)^2\left[\tr(D)I_r-\frac{3}{2\tr(D)}\Lambda(I_r-\Lambda)^2\right].
\end{align*}
Since $a(1-a)^2\le\frac{4}{27}$ for $0\le a\le1$, we have
\[
 \tr(D)I_r-\frac{3}{2\tr(D)}\Lambda(I_r-\Lambda)^2
 \succeq\left(\tr(D)-\frac{2}{9\tr(D)}\right)I_r
 \succeq\frac{1}{18}I_r.
\]
Together with $\mathcal R_\Lambda(T)\ge0$ from Lemma~\ref{lem:lu-auxiliary}, these inequalities prove \eqref{eq:lu-mixed} in this case.

Now let $0<\tr(D)<\frac{1}{2}$. Write $D=\tr(D)\Lambda-W$, where $W=\Lambda(\Lambda-(1-\tr(D))I_r)$, so that $TD+DT=\tr(D)(T\Lambda+\Lambda T)-(TW+WT)$. Completing the square gives
\begin{equation}\label{eq:lu-mixed-shift}
\begin{aligned}
 &\mathcal R_\Lambda(T) I_r+\tr(D)\left(\tr(\Lambda T)\right)^2I_r+\frac{2\tr(D)}{3}T\Lambda T-\tr(\Lambda T)(TD+DT)\\
 =&\frac{2\tr(D)}{3}\left(T-\frac{3\tr(\Lambda T)}{2}I_r\right)
       \Lambda\left(T-\frac{3\tr(\Lambda T)}{2}I_r\right)\\
 &+\mathcal R_\Lambda(T) I_r+\tr(\Lambda T)(TW+WT)+\tr(D)\left(\tr(\Lambda T)\right)^2\left(I_r-\frac{3}{2}\Lambda\right).
\end{aligned}
\end{equation}
The first term is positive semidefinite. For every unit vector $v=(v_i)\in\R^r$, \eqref{eq:lu-remainder-decomposition} gives
\[
\begin{aligned}
 &v^{\mathsf T}\left[\mathcal R_\Lambda(T)I_r
       +\tr(\Lambda T)(TW+WT)
       +\tr(D)\left(\tr(\Lambda T)\right)^2
         \left(I_r-\frac{3}{2}\Lambda\right)\right]v\\
 =&\sum_{i=1}^r\left[\mathcal R_{\Lambda,\mathrm{diag}}(T)
       +2\tr(\Lambda T)W_{ii}T_{ii}
       +\tr(D)\left(\tr(\Lambda T)\right)^2
         \left(1-\frac{3}{2}a_i\right)\right]v_i^2\\
 &+\sum_{1\le i<j\le r}\left[h_{ij}T_{ij}^2
       +2\tr(\Lambda T)(W_{ii}+W_{jj})T_{ij}v_i v_j\right].
\end{aligned}
\]
For the diagonal sum, we use $W_{ii}=a_i(a_i-1+\tr(D))$ and prove the following lower bound for each $i$:
\begin{equation}\label{eq:lu-single-index}
 \mathcal R_{\Lambda,\mathrm{diag}}(T)
 +2\tr(\Lambda T) a_i(a_i-1+\tr(D))T_{ii}\ge \tr(D) a_i\left(\tr(\Lambda T)\right)^2.
\end{equation}
The trace constraint gives
\[
 \sum_{j=1}^r a_j\left((2a_j-1)T_{jj}-\tr(\Lambda T)\right)=0.
\]
For each fixed $i$, applying the Cauchy--Schwarz inequality, we obtain
\[
\begin{aligned}
 a_i^2\left((2a_i-1)T_{ii}-\tr(\Lambda T)\right)^2=&\left[\sum_{\substack{j=1\\j\ne i}}^ra_j
       \left((2a_j-1)T_{jj}-\tr(\Lambda T)\right)\right]^2\\
 \le&(1-a_i)\sum_{\substack{j=1\\j\ne i}}^ra_j
       \left((2a_j-1)T_{jj}-\tr(\Lambda T)\right)^2.
\end{aligned}
\]
Combining this with the diagonal bound in the proof of Lemma~\ref{lem:lu-auxiliary}, we obtain
\[
\begin{aligned}
 \mathcal R_{\Lambda,\mathrm{diag}}(T)
 &\ge2\sum_{j=1}^r a_j
       \left((2a_j-1)T_{jj}-\tr(\Lambda T)\right)^2\\
 &\ge\left(2a_i+\frac{2a_i^2}{1-a_i}\right)
       \left((2a_i-1)T_{ii}-\tr(\Lambda T)\right)^2\\
 &=\frac{2a_i}{1-a_i}
       \left((2a_i-1)T_{ii}-\tr(\Lambda T)\right)^2.
\end{aligned}
\]
By \eqref{eq:diagonal-weight-bound}, $2a_i(1-a_i)\le\tr(D)<\frac{1}{2}$, so $a_i\ne\frac{1}{2}$. Completing the square in $T_{ii}$ gives
\begin{align*}
 &\mathcal R_{\Lambda,\mathrm{diag}}(T)
 +2\tr(\Lambda T)a_i(a_i-1+\tr(D))T_{ii}
 -\tr(D)a_i\left(\tr(\Lambda T)\right)^2\\
 \ge&\frac{2a_i}{1-a_i}
       \left((2a_i-1)T_{ii}-\tr(\Lambda T)\right)^2\\
 &+2a_i(a_i-1+\tr(D))\tr(\Lambda T)T_{ii}
 -\tr(D)a_i\left(\tr(\Lambda T)\right)^2\\
 =&\frac{2a_i}{1-a_i}
 \left[(2a_i-1)T_{ii}-\tr(\Lambda T)
 +\frac{(1-a_i)(a_i-1+\tr(D))}{2(2a_i-1)}
       \tr(\Lambda T)\right]^2\\
 &+a_i\left[
       \frac{2(a_i-1+\tr(D))}{2a_i-1}-\tr(D)
       -\frac{1-a_i}{2}
        \left(\frac{a_i-1+\tr(D)}{2a_i-1}\right)^2
       \right]\left(\tr(\Lambda T)\right)^2.
\end{align*}
The square is nonnegative. For fixed $a_i$, the bracket in the second term is a concave quadratic in $\tr(D)$ on $[2a_i(1-a_i),\frac{1}{2}]$ with endpoint values
\[
 \frac{(1-a_i)^2(3+a_i)}{2}>0,
 \qquad \frac{3+a_i}{8}>0.
\]
Thus the bracket is positive, proving \eqref{eq:lu-single-index}. Applying this estimate to the diagonal sum and using $a_i\le1$ and $\sum_{i=1}^r v_i^2=1$, we obtain
\[
\begin{aligned}
 &\sum_{i=1}^r\left[\mathcal R_{\Lambda,\mathrm{diag}}(T)
       +2\tr(\Lambda T)W_{ii}T_{ii}
       +\tr(D)\left(\tr(\Lambda T)\right)^2
         \left(1-\frac{3}{2}a_i\right)\right]v_i^2\\
 \ge&\tr(D)\left(\tr(\Lambda T)\right)^2
       \sum_{i=1}^r\left(1-\frac{a_i}{2}\right)v_i^2
 \ge\frac{\tr(D)}{2}\left(\tr(\Lambda T)\right)^2.
\end{aligned}
\]

For each $i<j$,
\[
 W_{ii}+W_{jj}
 =(1-a_i-a_j)(a_i^2+a_j^2)
       -(a_i+a_j)\sum_{\substack{k=1\\k\ne i,j}}^ra_k^2.
\]
Both quantities in this difference lie between $0$ and $(a_i+a_j)(1-a_i-a_j)$. Together with $\tr(D)\ge2(a_i+a_j)(1-a_i-a_j)$, this gives
\[
 |W_{ii}+W_{jj}|\le(a_i+a_j)(1-a_i-a_j)\le\frac{\tr(D)}{2}.
\]
From \eqref{eq:lu-offdiagonal-coefficient}, we have
\[
 h_{ij}=2(a_i+a_j)(1-a_i-a_j)-(W_{ii}+W_{jj})
       \ge |W_{ii}+W_{jj}|.
\]
For $h_{ij}>0$, these bounds and completion of the square give
\begin{align*}
 &h_{ij}T_{ij}^2
       +2\tr(\Lambda T)(W_{ii}+W_{jj})T_{ij}v_i v_j\\
 =&h_{ij}\left[T_{ij}
       +\frac{\tr(\Lambda T)(W_{ii}+W_{jj})}{h_{ij}}v_i v_j\right]^2
       -\left(\tr(\Lambda T)\right)^2\frac{(W_{ii}+W_{jj})^2}{h_{ij}}v_i^2v_j^2\\
 \ge&-\left(\tr(\Lambda T)\right)^2\frac{(W_{ii}+W_{jj})^2}{h_{ij}}v_i^2v_j^2\\
 \ge&-\frac{\tr(D)\left(\tr(\Lambda T)\right)^2}{2}v_i^2v_j^2.
\end{align*}
If $h_{ij}=0$, then $W_{ii}+W_{jj}=0$, and the final bound holds directly. Since $\sum_{1\le i<j\le r}v_i^2v_j^2=\frac{1-\sum_{i=1}^r v_i^4}{2}\le\frac{1}{2}$, summing over pairs gives
\[
\begin{aligned}
 &\sum_{1\le i<j\le r}\left[h_{ij}T_{ij}^2
       +2\tr(\Lambda T)(W_{ii}+W_{jj})T_{ij}v_i v_j\right]\\
 \ge&-\frac{\tr(D)\left(\tr(\Lambda T)\right)^2}{2}\sum_{1\le i<j\le r}v_i^2v_j^2
 \ge-\frac{\tr(D)}{4}\left(\tr(\Lambda T)\right)^2.
\end{aligned}
\]
Combining the diagonal and off diagonal bounds gives, for every unit vector $v$,
\[
\begin{aligned}
 &v^{\mathsf T}\left[\mathcal R_\Lambda(T)I_r
       +\tr(\Lambda T)(TW+WT)
       +\tr(D)\left(\tr(\Lambda T)\right)^2
         \left(I_r-\frac{3}{2}\Lambda\right)\right]v\\
 \ge&\frac{\tr(D)}{2}\left(\tr(\Lambda T)\right)^2
       -\frac{\tr(D)}{4}\left(\tr(\Lambda T)\right)^2
 =\frac{\tr(D)}{4}\left(\tr(\Lambda T)\right)^2\ge0.
\end{aligned}
\]
Combining this inequality with \eqref{eq:lu-mixed-shift}, we obtain \eqref{eq:lu-mixed}.
\end{proof}

\begin{proof}[Proof of Lemma~\ref{lem:matrix-parameter}]
The case $p=0$ follows from Lemma~\ref{lem:matrix}, so assume $0<p\le\frac{2}{3}$. The bound $\tr(D)\le2(1-a_i)$ gives $(D_p)_{ii}\ge(1-p)a_i(1-a_i)>0$, so $D_p$ is positive definite.

Set
\[
 T=\widetilde{T}-\frac{\tr(D\widetilde{T})}{\tr D}I_r
     =\widetilde{T}-\frac{p}{2}\tr(\Lambda \widetilde{T})I_r.
\]
Then $\tr(DT)=0$, and Lemma~\ref{lem:matrix} yields $2TDT\preceq Q_\Lambda(T)I_r$. Since $\tr\Lambda=1$, we also have $\tr(\Lambda T)=(1-\frac{p}{2})\tr(\Lambda \widetilde{T})$, and hence
\[
 \widetilde{T}=T+\frac{p}{2-p}\tr(\Lambda T)I_r.
\]
Using the definition of $D_p$ and expanding $\widetilde{T}D\widetilde{T}$, we obtain
\begin{equation}\label{eq:lu-shift-product}
\begin{aligned}
 \widetilde{T}D_p\widetilde{T}
 &=\widetilde{T}D\widetilde{T}-\frac{p\tr(D)}{2}\widetilde{T}\Lambda \widetilde{T}\\
 &=TDT+\frac{p}{2-p}\tr(\Lambda T)(TD+DT)+\frac{p^2}{(2-p)^2}\left(\tr(\Lambda T)\right)^2D
       -\frac{p\tr(D)}{2}\widetilde{T}\Lambda \widetilde{T}.
\end{aligned}
\end{equation}

Substituting the above expression for $\widetilde{T}$ into $Q_{\Lambda,p}$ and using $\tr(DT)=0$ and the definition of $\mathcal R_\Lambda(T)$, we obtain
\[
\begin{aligned}
 2(2-p)Q_{\Lambda,p}(\widetilde{T})
 =&2(2-p)Q_\Lambda(T)-3p(2-p)\tr(D)\tr(\Lambda T^2)
       +p(4-3p)\tr(D)\left(\tr(\Lambda T)\right)^2\\
 =&(2-p)(2-3p)Q_\Lambda(T)+3p(2-p)\mathcal R_\Lambda(T)
       +p(4-3p)\tr(D)\left(\tr(\Lambda T)\right)^2.
\end{aligned}
\]
Substituting the scalar shift into \eqref{eq:lu-shift-product} and combining it with the preceding identity, we obtain
\[
\begin{aligned}
 &2(2-p)^2\left[Q_{\Lambda,p}(\widetilde{T})I_r-(2-3p)\widetilde{T}D_p\widetilde{T}\right]\\
 =&(2-p)^2(2-3p)\left[Q_\Lambda(T)I_r-2TDT\right]+3p(2-p)^2\mathcal R_\Lambda(T)I_r-2p^2(2-3p)\left(\tr(\Lambda T)\right)^2D\\
 &-2p(2-p)(2-3p)\tr(\Lambda T)(TD+DT)+p(2-p)(4-3p)\tr(D)\left(\tr(\Lambda T)\right)^2I_r\\
 &+p(2-3p)\tr(D)
       \left((2-p)T+p\tr(\Lambda T)I_r\right)
       \Lambda\left((2-p)T+p\tr(\Lambda T)I_r\right).
\end{aligned}
\]
Subtracting $2p(2-p)(2-3p)$ times the matrix in \eqref{eq:lu-mixed} cancels the term involving $TD+DT$. Completing the remaining products involving $T$ to a square, we obtain
\[
\begin{aligned}
 &p(2-3p)\tr(D)
       \left((2-p)T+p\tr(\Lambda T)I_r\right)
       \Lambda\left((2-p)T+p\tr(\Lambda T)I_r\right)\\
 &-\frac{4p(2-p)(2-3p)}{3}\tr(D)T\Lambda T\\
 =&\frac{p(2-p)}{3}\tr(D)
       \left((2-3p)T+3p\tr(\Lambda T)I_r\right)
       \Lambda\left((2-3p)T+3p\tr(\Lambda T)I_r\right)\\
 &-4p^3\tr(D)\left(\tr(\Lambda T)\right)^2\Lambda.
\end{aligned}
\]
Substituting this identity and collecting the remaining diagonal terms gives
\begin{equation}\label{eq:lu-matrix-certificate}
\begin{aligned}
 &2(2-p)^2\left[Q_{\Lambda,p}(\widetilde{T})I_r-(2-3p)\widetilde{T}D_p\widetilde{T}\right]\\
 =&(2-p)^2(2-3p)\left[Q_\Lambda(T)I_r-2TDT\right]+p(2-p)(2+3p)\mathcal R_\Lambda(T)I_r\\
 &+2p(2-p)(2-3p)\left[\left(\mathcal R_\Lambda(T)
          +\tr(D)\left(\tr(\Lambda T)\right)^2\right)I_r+\frac{2\tr(D)}{3}T\Lambda T-\tr(\Lambda T)(TD+DT)\right]\\
 &+\frac{p(2-p)}{3}\tr(D)
       \left((2-3p)T+3p\tr(\Lambda T)I_r\right)
       \Lambda\left((2-3p)T+3p\tr(\Lambda T)I_r\right)\\
 &+p^2\left(\tr(\Lambda T)\right)^2
       \left[3(2-p)\tr(D)I_r-2(2-3p)D-4p\tr(D)\Lambda\right].
\end{aligned}
\end{equation}
For $0<p\le\frac{2}{3}$, all five scalar coefficients are nonnegative. Lemmas~\ref{lem:matrix}, \ref{lem:lu-auxiliary}, and~\ref{lem:lu-mixed} show that the first three terms are positive semidefinite, respectively. The fourth is positive semidefinite since $T$ is symmetric and $\Lambda$ is positive definite. For the last term, \eqref{eq:diagonal-weight-bound} and $\Lambda\preceq I_r$ give
\[
 3(2-p)\tr(D)I_r-2(2-3p)D-4p\tr(D)\Lambda
 \succeq4(1-p)\tr(D)I_r\succeq0.
\]
Thus the right-hand side of \eqref{eq:lu-matrix-certificate} is positive semidefinite. Dividing by $2(2-p)^2>0$ proves \eqref{eq:lu-matrix}.
\end{proof}

\section{Differential inequalities, compactness, and spectral estimates}\label{sec:estimates}

For the following two lemmas, let $x:M^n\to\Sph^{n+q}$ be a minimal immersion with flat normal bundle. At a point where $E>0$, use the nonzero shape operators to define
\begin{equation}\label{eq:geometric-matrices}
 \Lambda=\frac{1}{S}\diag\left(|A^1|^2,\ldots,|A^r|^2\right),
 \qquad D=\Lambda(I_r-\Lambda).
\end{equation}

By Lemma~\ref{lem:matrix} and orthogonal projection, we obtain the following estimates.

\begin{lemma}\label{lem:directional-estimate}
Let $(B_{i\alpha})$ be a real $n\times r$ array satisfying $\sum_{i=1}^n B_{i\alpha}=0$ for each $\alpha=1,\ldots,r$. Suppose that the contraction matrix $T=(T_{\alpha\beta})$, defined by
\[
 T_{\alpha\beta}=\frac{\sum_{i=1}^n A^\alpha_{ii}B_{i\beta}}
 {|A^\alpha|\,|A^\beta|},
\]
is symmetric. Then
\begin{equation}\label{eq:component-projection-energy}
  \sum_{\alpha=1}^r\sum_{i=1}^n\left(S-|A^\alpha|^2\right)B_{i\alpha}^2
       -S^2\tr(T^2D)\ge0.
\end{equation}
If $\tr(DT)=0$, then for each $k=1,\ldots,n$,
\begin{equation}\label{eq:component-row-bound}
 \begin{aligned}
 \sum_{\alpha=1}^r\left(S-|A^\alpha|^2\right)B_{k\alpha}^2\le\frac{n-1}{n}\left\{
       \frac{S^2}{2}Q_{\Lambda}(T)
       +\sum_{\alpha=1}^r\sum_{i=1}^n\left(S-|A^\alpha|^2\right)B_{i\alpha}^2
       -S^2\tr(T^2D)\right\}.
 \end{aligned}
\end{equation}

\end{lemma}

\begin{proof}
Fix a row index $k$. The vectors $\left(\frac{A^\alpha_{ii}}{|A^\alpha|}\right)_{i=1}^n$ are orthonormal by \eqref{eq:normal-orthogonality}. From the definition of $T$, we have
\[
 \sum_{i=1}^n\frac{A^\gamma_{ii}}{|A^\gamma|}
 \left(B_{i\alpha}
       -|A^\alpha|\sum_{\beta=1}^r\frac{A^\beta_{ii}}{|A^\beta|}T_{\beta\alpha}\right)=0.
\]
Orthogonal projection and the definitions of $D$ and $T$ give
\[
 \begin{aligned}
 \sum_{\alpha=1}^r\sum_{i=1}^n\left(S-|A^\alpha|^2\right)
 \left(B_{i\alpha}
       -|A^\alpha|\sum_{\beta=1}^r\frac{A^\beta_{ii}}{|A^\beta|}T_{\beta\alpha}\right)^2=\sum_{\alpha=1}^r\sum_{i=1}^n\left(S-|A^\alpha|^2\right)B_{i\alpha}^2
       -S^2\tr(T^2D)\ge0.
 \end{aligned}
\]
This proves \eqref{eq:component-projection-energy}.

Assume $\tr(DT)=0$ and let $v=(v_\alpha)\in\R^r$ be any unit vector. From $\sum_{i=1}^n B_{i\alpha}=0$ and the Cauchy--Schwarz inequality, we obtain
\[
\begin{aligned}
 \left(\sum_{\alpha=1}^r\sqrt{S-|A^\alpha|^2}\, v_\alpha B_{k\alpha}\right)^2=&\left[\sum_{i=1}^n\left(\delta_{ik}-\frac{1}{n}\right)
       \sum_{\alpha=1}^r\sqrt{S-|A^\alpha|^2}\,
                    v_\alpha B_{i\alpha}\right]^2\\
 \le&\frac{n-1}{n}\sum_{i=1}^n
       \left(\sum_{\alpha=1}^r\sqrt{S-|A^\alpha|^2}\,
                    v_\alpha B_{i\alpha}\right)^2.
\end{aligned}
\]

Since $T\sqrt D$ and its transpose $\sqrt D T$ have the same operator norm, Lemma~\ref{lem:matrix} gives
\[
 |T\sqrt D\,v|^2\le\frac{1}{2}Q_{\Lambda}(T).
\]
Using the above orthogonal decomposition and the Cauchy--Schwarz inequality, we obtain
\begin{align*}
 &\sum_{i=1}^n\left(\sum_{\alpha=1}^r\sqrt{S-|A^\alpha|^2}\,
                    v_\alpha B_{i\alpha}\right)^2\\
                    =&S^2|T\sqrt D\,v|^2+\sum_{i=1}^n\left\{\sum_{\alpha=1}^r\sqrt{S-|A^\alpha|^2}\,v_\alpha
       \left(B_{i\alpha}
       -|A^\alpha|\sum_{\beta=1}^r\frac{A^\beta_{ii}}{|A^\beta|}T_{\beta\alpha}\right)\right\}^2\\
 \le&\frac{S^2}{2}Q_{\Lambda}(T)
       +\sum_{\alpha=1}^r\sum_{i=1}^n\left(S-|A^\alpha|^2\right)B_{i\alpha}^2
       -S^2\tr(T^2D).
\end{align*}
By $Q_\Lambda(T)\ge0$ and \eqref{eq:component-projection-energy}, we may assume $(B_{k1},\ldots,B_{kr})\ne0$ and choose
\[
 v_\alpha=\frac{\sqrt{S-|A^\alpha|^2}\,B_{k\alpha}}
 {\sqrt{\sum_{\beta=1}^r(S-|A^\beta|^2)B_{k\beta}^2}}.
\]
Then $|v|=1$, and substituting this vector into the above estimates gives
\[
\begin{aligned}
 \sum_{\alpha=1}^r(S-|A^\alpha|^2)B_{k\alpha}^2
 &=\left(\sum_{\alpha=1}^r\sqrt{S-|A^\alpha|^2}\,v_\alpha B_{k\alpha}\right)^2\\
 &\le\frac{n-1}{n}\left\{
       \frac{S^2}{2}Q_\Lambda(T)
       +\sum_{\alpha=1}^r\sum_{i=1}^n(S-|A^\alpha|^2)B_{i\alpha}^2
       -S^2\tr(T^2D)\right\}.
\end{aligned}
\]
This proves \eqref{eq:component-row-bound}.

\end{proof}

\begin{lemma}\label{lem:coercive}
For $n\ge3$, we have
\begin{equation}\label{eq:coercive}
\begin{aligned}
&\sum_{\alpha=1}^q(S-|A^\alpha|^2)|\nabla A^\alpha|^2
+\frac{1}{2}|\nabla S|^2
-2\sum_{\alpha,\beta=1}^q\sum_{k=1}^n
\left(\sum_{i,j=1}^nA^\alpha_{ij}A^\beta_{ij;k}\right)^2\\
\ge&\frac{E}{6n(n-1)S}|\nabla A|^2
-\frac{2n+1}{16}\frac{|\nabla E|^2}{E}.
\end{aligned}
\end{equation}
\end{lemma}
\begin{proof}
For the nonzero shape operators, the Codazzi equation gives
\begin{equation}\label{eq:codazzi-square-decomposition}
\sum_{i,j,l=1}^n(A^\alpha_{ij;l})^2
=3\sum_{i,k=1}^n(A^\alpha_{ii;k})^2-2\sum_{k=1}^n(A^\alpha_{kk;k})^2
+\sum_{\substack{1\le i,j,l\le n\\i\ne j,\ j\ne l,\ l\ne i}}(A^\alpha_{ij;l})^2.
\end{equation}

Fix a tangent index $k$ and define the $r\times r$ matrix $F=(F_{\alpha\beta})$ by
\[
 F_{\alpha\beta}=\frac{\sum_{i=1}^n A^\alpha_{ii}A^\beta_{ii;k}}
 {|A^\alpha|\,|A^\beta|}.
\]
By Lemma~\ref{lem:symmetry}, $F$ is symmetric. From \eqref{eq:defect} and the definition of $F$, we have
\begin{equation}\label{eq:geometric-traces}
 \tr D=\frac{E}{S^2},\qquad
 \tr(\Lambda F)=\frac{\nabla_kS}{2S},\qquad
 \tr(DF)=\frac{\nabla_kE}{4S^2}.
\end{equation}
Using the above identities and the definition of $Q_\Lambda$, we obtain
\[
\begin{aligned}
 &3\sum_{\alpha=1}^r\sum_{i=1}^n(S-|A^\alpha|^2)(A^\alpha_{ii;k})^2
 -2\sum_{\alpha=1}^r(S-|A^\alpha|^2)(A^\alpha_{kk;k})^2+\frac{1}{2}(\nabla_kS)^2
 -2\sum_{\alpha,\beta=1}^r\left(\sum_{i=1}^n A^\alpha_{ii}A^\beta_{ii;k}\right)^2\\
 =&S^2\left[Q_\Lambda(F)-3\tr(F^2D)\right]
 +3\sum_{\alpha=1}^r\sum_{i=1}^n(S-|A^\alpha|^2)(A^\alpha_{ii;k})^2
 -2\sum_{\alpha=1}^r(S-|A^\alpha|^2)(A^\alpha_{kk;k})^2.
\end{aligned}
\]

By \eqref{eq:geometric-traces}, the matrix $T=F-\frac{\tr(DF)}{\tr D}I_r=F-\frac{\nabla_kE}{4E}I_r$ satisfies $\tr(DT)=0$. Applying Young's inequality, we obtain
\[
\begin{aligned}
 \sum_{\alpha=1}^r(S-|A^\alpha|^2)(A^\alpha_{kk;k})^2\le&\left(1+\frac{1}{2(n-1)}\right)
 \sum_{\alpha=1}^r(S-|A^\alpha|^2)
 \left(A^\alpha_{kk;k}-\frac{\nabla_kE}{4E}A^\alpha_{kk}\right)^2\\
 &+\frac{(2n-1)(\nabla_kE)^2}{16E^2}
 \sum_{\alpha=1}^r(S-|A^\alpha|^2)(A^\alpha_{kk})^2.
\end{aligned}
\]

From $\tr A^\alpha=0$ and \eqref{eq:trace-derivative}, we have
\[
 \sum_{i=1}^n\left(A^\alpha_{ii;k}
 -\frac{\nabla_kE}{4E}A^\alpha_{ii}\right)=0.
\]
From \eqref{eq:normal-orthogonality}, we also have
\[
\begin{aligned}
 \frac{\sum_{i=1}^n A^\alpha_{ii}\left(A^\beta_{ii;k}
       -\frac{\nabla_kE}{4E}A^\beta_{ii}\right)}
      {|A^\alpha|\,|A^\beta|}
 =F_{\alpha\beta}-\frac{\nabla_kE}{4E}
       \frac{\sum_{i=1}^n A^\alpha_{ii}A^\beta_{ii}}
            {|A^\alpha|\,|A^\beta|}=F_{\alpha\beta}-\frac{\nabla_kE}{4E}\delta_{\alpha\beta}
 =T_{\alpha\beta}.
\end{aligned}
\]
By the above identities, Lemma~\ref{lem:directional-estimate} applies, since $T$ is symmetric and $\tr(DT)=0$. The subtracted vectors belong to the projection subspace, so the remainder in \eqref{eq:component-projection-energy} is unchanged, and we obtain
\begin{equation}\label{eq:row-input}
\begin{aligned}
&\sum_{\alpha=1}^r(S-|A^\alpha|^2)
\left(A^\alpha_{kk;k}-\frac{\nabla_kE}{4E}A^\alpha_{kk}\right)^2\\
\le&\frac{n-1}{n}\left\{
\frac{S^2}{2}Q_\Lambda(T)
+\sum_{\alpha=1}^r\sum_{i=1}^n(S-|A^\alpha|^2)(A^\alpha_{ii;k})^2
-S^2\tr(F^2D)\right\}.
\end{aligned}
\end{equation}

Applying Bessel's inequality to $(\delta_{ik}-\frac{1}{n})_{i=1}^n$ and the orthonormal family used in the proof of Lemma~\ref{lem:directional-estimate}, we obtain
\[
\begin{aligned}
 \sum_{\alpha=1}^r\frac{(A^\alpha_{kk})^2}{|A^\alpha|^2}=\sum_{\alpha=1}^r\left[\sum_{i=1}^n\left(\delta_{ik}-\frac{1}{n}\right)
       \frac{A^\alpha_{ii}}{|A^\alpha|}\right]^2\le\sum_{i=1}^n\left(\delta_{ik}-\frac{1}{n}\right)^2
 =\frac{n-1}{n}.
\end{aligned}
\]
From \eqref{eq:defect} and \eqref{eq:normal-orthogonality}, we obtain
\[
 |A^\alpha|^2(S-|A^\alpha|^2)
 =\sum_{\substack{\beta=1\\\beta\ne\alpha}}^r|A^\alpha|^2|A^\beta|^2
 \le\frac{E}{2}.
\]
Combining these estimates, we obtain
\begin{equation}\label{eq:shape-row-bound}
\begin{aligned}
\sum_{\alpha=1}^r(S-|A^\alpha|^2)(A^\alpha_{kk})^2
&\le\frac{E}{2}\sum_{\alpha=1}^r\frac{(A^\alpha_{kk})^2}{|A^\alpha|^2}
\le\frac{n-1}{2n}E.
\end{aligned}
\end{equation}

Combining the above estimates with \eqref{eq:component-projection-energy} and using $S^2Q_\Lambda(F)=S^2Q_\Lambda(T)+\frac{5(\nabla_kE)^2}{16E}$, we obtain
\begin{equation}\label{eq:row}
\begin{aligned}
&S^2\left[Q_\Lambda(F)-3\tr(F^2D)\right]
+3\sum_{\alpha=1}^r\sum_{i=1}^n(S-|A^\alpha|^2)(A^\alpha_{ii;k})^2-2\sum_{\alpha=1}^r(S-|A^\alpha|^2)(A^\alpha_{kk;k})^2\\
\ge&\frac{S^2}{2n}Q_\Lambda(T)
+\sum_{\alpha=1}^r\sum_{i=1}^n(S-|A^\alpha|^2)(A^\alpha_{ii;k})^2
-S^2\tr(F^2D)-\frac{n(\nabla_kE)^2}{8E}.
\end{aligned}
\end{equation}

Applying Lemma~\ref{lem:matrix} to $T$, we obtain
\begin{equation}\label{eq:matrix-input}
2TDT\preceq Q_\Lambda(T)I_r.
\end{equation}
From the definition of $T$ and \eqref{eq:geometric-traces}, we obtain
\begin{equation}\label{eq:shift-energy}
S^2\tr(F^2D)-S^2\tr(T^2D)
=S^2\left(\frac{\nabla_kE}{4E}\right)^2\tr D
=\frac{(\nabla_kE)^2}{16E}.
\end{equation}
Taking the trace in \eqref{eq:matrix-input} and using \eqref{eq:shift-energy}, we obtain
\[
 \frac{S^2}{2n}Q_\Lambda(T)
 \ge\frac{S^2}{nr}\tr(T^2D)
 =\frac{S^2}{nr}\tr(F^2D)-\frac{(\nabla_kE)^2}{16nrE}.
\]
Together with \eqref{eq:component-projection-energy} and $r\le n-1$, this gives
\[
\begin{aligned}
 &\frac{S^2}{2n}Q_\Lambda(T)
 +\sum_{\alpha=1}^r\sum_{i=1}^n(S-|A^\alpha|^2)(A^\alpha_{ii;k})^2
 -S^2\tr(F^2D)-\frac{n(\nabla_kE)^2}{8E}\\
 \ge&\frac{1}{nr}
 \sum_{\alpha=1}^r\sum_{i=1}^n(S-|A^\alpha|^2)(A^\alpha_{ii;k})^2
 -\left(2n+\frac{1}{nr}\right)\frac{(\nabla_kE)^2}{16E}\\
 \ge&\frac{1}{n(n-1)}
 \sum_{\alpha=1}^r\sum_{i=1}^n(S-|A^\alpha|^2)(A^\alpha_{ii;k})^2
 -\frac{(2n+1)(\nabla_kE)^2}{16E}.
\end{aligned}
\]

For the $r$ nonzero shape operators, \eqref{eq:codazzi-square-decomposition} bounds the weighted derivative squares with at least two equal tangent indices by three times the weighted diagonal derivative squares. The terms with three distinct indices retain their full weights. By Lemma~\ref{lem:symmetry}, each zero shape operator has vanishing contraction terms and contributes $S|\nabla A^\alpha|^2$. Summing \eqref{eq:row} over $k$, combining the above estimates and including these terms, we obtain
\[
\begin{aligned}
&\sum_{\alpha=1}^q(S-|A^\alpha|^2)|\nabla A^\alpha|^2
+\frac{1}{2}|\nabla S|^2
-2\sum_{\alpha,\beta=1}^q\sum_{k=1}^n
\left(\sum_{i,j=1}^nA^\alpha_{ij}A^\beta_{ij;k}\right)^2\\
\ge&\frac{1}{3n(n-1)}
\sum_{\alpha=1}^q(S-|A^\alpha|^2)|\nabla A^\alpha|^2
-\frac{2n+1}{16}\frac{|\nabla E|^2}{E}.
\end{aligned}
\]
By \eqref{eq:defect}, we have for every $\alpha$
\[
E\le S^2-|A^\alpha|^4
\le2S(S-|A^\alpha|^2).
\]
This proves \eqref{eq:coercive}.
\end{proof}

For $n\ge3$, applying the chain rule to $E^n$ and using \eqref{eq:E-evolution} and \eqref{eq:coercive}, we obtain on $\{E>0\}$
\begin{align}
\Delta(E^n)\ge&
\frac{2E^n}{3(n-1)S}|\nabla A|^2
+\frac{n(2n-5)}{4} E^{n-2}|\nabla E|^2 \notag\\
&+2n(2n-S)E^n
+12nE^{n-1}\sum_{1\le\alpha<\beta<\gamma\le q}|A^\alpha|^2|A^\beta|^2|A^\gamma|^2.
\label{eq:power}
\end{align}
Since $E>0$ implies $S>0$, \eqref{eq:power} gives $\Delta(E^n)\ge0$ on $\{E>0\}$ whenever $S\le2n$. At zeros of $E$, the chain rule gives $\Delta(E^n)=0$ because $n\ge3$. Thus $E^n$ is smooth and subharmonic on all of $M$ under $S\le2n$.

Using Lemma~\ref{lem:matrix-parameter}, we obtain a derivative estimate at critical points of $E/S^{2/(3n)}$.

\begin{lemma}\label{lem:lu-critical-derivative}
Let $x:M^n\to\Sph^{n+q}$ be a minimal immersion with flat normal bundle, where $n\ge3$, and set $p=\frac{2}{3n}$. At a critical point of $E/S^p$ where $E>0$, we have
\begin{equation}\label{eq:lu-critical-derivative}
\begin{aligned}
\sum_{\alpha=1}^q\left(S-|A^\alpha|^2-\frac{pE}{2S}\right)
                   |\nabla A^\alpha|^2
 +\left(\frac{1}{2}+\frac{p(1-p)E}{4S^2}\right)|\nabla S|^2-2\sum_{\alpha,\beta=1}^q\sum_{k=1}^n
       \left(\sum_{i,j=1}^nA^\alpha_{ij}A^\beta_{ij;k}\right)^2
 \ge0.
\end{aligned}
\end{equation}
\end{lemma}
\begin{proof}
From $\nabla(E/S^p)=0$, we have $\nabla E=\frac{pE}{S}\nabla S$. For each tangent index $k$, take $\Lambda,D,F$ as in the proof of Lemma~\ref{lem:coercive}. Then \eqref{eq:geometric-traces} gives
\[
 \tr(D_pF)=\frac{\nabla_kE-\frac{pE}{S}\nabla_kS}{4S^2}=0.
\]
Since $D_p$ is positive definite, for $1\le\alpha\le r$ we have
\[
 S-|A^\alpha|^2-\frac{pE}{2S}>0.
\]

Repeating the orthogonal projection in the proof of Lemma~\ref{lem:directional-estimate} with these positive weights gives
\[
 \sum_{\alpha=1}^r\sum_{i=1}^n
 \left(S-|A^\alpha|^2-\frac{pE}{2S}\right)(A^\alpha_{ii;k})^2
 -S^2\tr(F^2D_p)\ge0.
\]
Using \eqref{eq:lu-q-p} and \eqref{eq:geometric-traces} together with the definition of $F$, we obtain
\[
\begin{aligned}
 &\left(\frac{1}{2}+\frac{p(1-p)E}{4S^2}\right)|\nabla_kS|^2
 -2\sum_{\alpha,\beta=1}^r
 \left(\sum_{i=1}^n A^\alpha_{ii}A^\beta_{ii;k}\right)^2\\
 =&S^2\left\{\left[2+p(1-p)\tr(D)\right]\left(\tr(\Lambda F)\right)^2
       -2\tr(\Lambda F\Lambda F)\right\}\\
 =&S^2\left[Q_{\Lambda,p}(F)-3\tr(F^2D_p)\right].
\end{aligned}
\]
Combining this identity with \eqref{eq:codazzi-square-decomposition}, we obtain for each $k$
\[
\begin{aligned}
 &\sum_{\alpha=1}^r\left(S-|A^\alpha|^2-\frac{pE}{2S}\right)
 \left[3\sum_{i=1}^n(A^\alpha_{ii;k})^2-2(A^\alpha_{kk;k})^2\right]\\
 &+\left(\frac{1}{2}+\frac{p(1-p)E}{4S^2}\right)|\nabla_kS|^2
 -2\sum_{\alpha,\beta=1}^r
 \left(\sum_{i=1}^n A^\alpha_{ii}A^\beta_{ii;k}\right)^2\\
 =&S^2Q_{\Lambda,p}(F)
 +3\left[\sum_{\alpha=1}^r\sum_{i=1}^n
 \left(S-|A^\alpha|^2-\frac{pE}{2S}\right)(A^\alpha_{ii;k})^2
 -S^2\tr(F^2D_p)\right]\\
 &-2\sum_{\alpha=1}^r
 \left(S-|A^\alpha|^2-\frac{pE}{2S}\right)(A^\alpha_{kk;k})^2.
\end{aligned}
\]
Applying Lemma~\ref{lem:matrix-parameter} in the projection argument of Lemma~\ref{lem:directional-estimate} and using $1-\frac{3p}{2}=\frac{n-1}{n}$, we obtain
\begin{equation}\label{eq:lu-weighted-row}
\begin{aligned}
 &\sum_{\alpha=1}^r
 \left(S-|A^\alpha|^2-\frac{pE}{2S}\right)(A^\alpha_{kk;k})^2\\
 \le& \frac{n-1}{n}\left[
 S^2\lambda_{\max}(FD_pF)
 +\sum_{\alpha=1}^r\sum_{i=1}^n
 \left(S-|A^\alpha|^2-\frac{pE}{2S}\right)(A^\alpha_{ii;k})^2
 -S^2\tr(F^2D_p)\right]\\
 \le& \frac{S^2}{2}Q_{\Lambda,p}(F)+\frac{n-1}{n}\left[
 \sum_{\alpha=1}^r\sum_{i=1}^n
 \left(S-|A^\alpha|^2-\frac{pE}{2S}\right)(A^\alpha_{ii;k})^2
 -S^2\tr(F^2D_p)\right].
\end{aligned}
\end{equation}
where $\lambda_{\max}$ denotes the largest eigenvalue.

Combining the preceding identity with \eqref{eq:lu-weighted-row}, we obtain
\[
\begin{aligned}
 &\sum_{\alpha=1}^r
 \left(S-|A^\alpha|^2-\frac{pE}{2S}\right)
 \left[3\sum_{i=1}^n(A^\alpha_{ii;k})^2-2(A^\alpha_{kk;k})^2\right]\\
 &+\left(\frac{1}{2}+\frac{p(1-p)E}{4S^2}\right)|\nabla_kS|^2
 -2\sum_{\alpha,\beta=1}^r
 \left(\sum_{i=1}^n A^\alpha_{ii}A^\beta_{ii;k}\right)^2\\
 \ge&\frac{n+2}{n}\left[
 \sum_{\alpha=1}^r\sum_{i=1}^n
 \left(S-|A^\alpha|^2-\frac{pE}{2S}\right)(A^\alpha_{ii;k})^2
 -S^2\tr(F^2D_p)\right]\ge0.
\end{aligned}
\]

The terms with three distinct tangent indices retain their positive weights. For a zero shape operator, the contraction terms vanish by Lemma~\ref{lem:symmetry}, leaving $S\left(1-\frac{p\tr(D)}{2}\right)|\nabla A^\alpha|^2\ge0$. Summing over $k$ and restoring these terms proves \eqref{eq:lu-critical-derivative}.
\end{proof}

For complete immersions, we use the following compactness lemma.

\begin{lemma}\label{lem:compactness}
Let $x_j:(M_j^n,p_j)\to\Sph^{n+q}$ be a sequence of minimal immersions of complete manifolds into a fixed round sphere. Suppose that their second fundamental forms are uniformly bounded. Then a subsequence converges smoothly on compact subsets in the pointed immersion sense to a complete minimal immersion $x_\infty:(M_\infty^n,p_\infty)\to\Sph^{n+q}$. If each $x_j$ has flat normal bundle, then so does $x_\infty$.
\end{lemma}
\begin{proof}
Since the ambient sphere is fixed and $|A|$ is uniformly bounded, the intrinsic injectivity radii have a uniform positive lower bound by Chen--Yin \cite[Theorem~2.1]{ChenYin2007}. The minimal immersions are stationary solutions of the mean curvature flow on $[0,1]$. By the local derivative estimates in \cite[Theorem~3.2]{ChenYin2007} at time $1$, we obtain uniform bounds for all covariant derivatives of $A$ on intrinsic balls of a fixed radius.

After passing to a subsequence, the images of the base points converge in the sphere. Applying Smith's compactness theorem \cite[Theorem~1.2]{Smith2007} at each order and taking a diagonal subsequence, we obtain a complete pointed smooth limit. On each relatively compact open subset of the limit, the convergence maps are diffeomorphisms onto their images for all sufficiently large indices, and the pulled-back metrics and composed immersion maps converge smoothly \cite[Section~2.2]{Smith2007}. Minimality and flatness of the normal bundle pass to the limit by smooth convergence.
\end{proof}

\begin{lemma}\label{lem:limit}
Let $x:M^n\to\Sph^{n+q}$ be a minimal immersion of a complete connected manifold with flat normal bundle, where $n\ge3$. Assume $0\le S\le2n$. If $\sup_M E>0$, then $\sup_M E=2n^2$. Moreover, if $E(p_j)\to \sup_M E$, a subsequence of the immersions with base points $p_j$ converges smoothly to a complete pointed limit satisfying
\[
S_\infty=2n,\qquad E_\infty=2n^2,\qquad
\nabla A_\infty=0,\qquad
\operatorname{Ric}_\infty\ge0.
\]
The Gram matrix $\G_\infty$ has exactly two nonzero eigenvalues, both equal to $n$. Consequently, every such $M$ satisfies $E\le2n^2$.
\end{lemma}
\begin{proof}
Since $E\le S^2\le4n^2$, its supremum is finite. Choose $p_j$ with $E(p_j)\to \sup_M E$, and apply Lemma~\ref{lem:compactness}, taking the connected component of the limit containing the base point. On this complete limit,
\[
E_\infty\le \sup_M E,\qquad E_\infty(p_\infty)=\sup_M E>0.
\]
The function $E_\infty^n$ is subharmonic by \eqref{eq:power} and attains its maximum at $p_\infty$, so the strong maximum principle gives $E_\infty\equiv \sup_M E>0$. From \eqref{eq:power}, we obtain
\[
S_\infty=2n,\qquad \nabla A_\infty=0,\qquad
\sum_{1\le\alpha<\beta<\gamma\le q}
|A_\infty^\alpha|^2|A_\infty^\beta|^2|A_\infty^\gamma|^2=0.
\]
By \eqref{eq:simons-flat}, we have $\tr(\G_\infty^2)=nS_\infty=2n^2$, which implies $\sup_M E=2n^2$. The vanishing triple sum implies $\rank\G_\infty\le2$. Since $E_\infty>0$, there are exactly two nonzero eigenvalues. Their sum is $2n$ and their sum of squares is $2n^2$, so both are equal to $n$.

Since the two nonzero shape operators have squared norm $n$, substituting $S_\infty=2n$ and $E_\infty=2n^2$ into \eqref{eq:shape-row-bound} gives
\[
\sum_{\alpha=1}^2(A^\alpha_{kk})^2\le n-1.
\]
The Ricci tensor is diagonal in this tangent frame, and the Gauss equation gives $\operatorname{Ric}_\infty(e_k,e_k) =n-1-\sum_{\alpha=1}^2(A^\alpha_{kk})^2\ge0$.
\end{proof}

\begin{lemma}\label{lem:spectral}
Let $x:M^n\to\Sph^{n+q}$ be a minimal immersion of a complete connected manifold with flat normal bundle, where $n\ge3$. Assume $S\le2n$ and $E\not\equiv2n^2$, and set $\Omega=\{E>n^2\}$. Every $\varphi\in C_c^\infty(\Omega)$ satisfies
\begin{equation}\label{eq:spectral}
\lambda\int_M\varphi^2\le\int_M|\nabla\varphi|^2,
\end{equation}
where $\lambda=\frac{e^{-\frac{2n+1}{2}}}{12(n-1)}>0$.
\end{lemma}
\begin{proof}
On $\Omega$, the bounds $E>n^2$ and $S\le2n$ give $\frac{E}{6n(n-1)S}\ge\frac{1}{12(n-1)}$.
Set
\[
v=2n^2-E-\frac{2n-S}{24(n-1)}.
\]
Combining \eqref{eq:simons-flat} and \eqref{eq:E-evolution} with \eqref{eq:coercive}, we obtain
\begin{equation}
\begin{aligned}
\Delta v
&=-\Delta E+\frac{|\nabla A|^2+E+nS-S^2}{12(n-1)}\\
&\le-\frac{|\nabla A|^2}{4(n-1)}
+\frac{2n+1}{4}\frac{|\nabla E|^2}{E}
-\frac{2n^2-E}{12(n-1)}
+\left(\frac{n}{4(n-1)}-2E\right)(2n-S)\\
&\quad-\frac{(2n-S)^2}{12(n-1)}
-12\sum_{1\le\alpha<\beta<\gamma\le q}|A^\alpha|^2|A^\beta|^2|A^\gamma|^2.
\end{aligned}
\label{eq:v-exact}
\end{equation}
Using $\nabla E=\frac{\nabla S}{24(n-1)}-\nabla v$ and $|\nabla S|^2\le4S|\nabla A|^2$, we obtain by Young's inequality
\[
\begin{aligned}
\frac{2n+1}{4}\frac{|\nabla E|^2}{E}
&\le\frac{(2n+1)S}{288(n-1)^2E}|\nabla A|^2
    +\frac{2n+1}{2E}|\nabla v|^2\\
&\le\frac{|\nabla A|^2}{12(n-1)}+\frac{2n+1}{2n^2}|\nabla v|^2.
\end{aligned}
\]
Substituting this bound in \eqref{eq:v-exact} and using $\frac{n}{4(n-1)}-2E\le-n^2$ on $\Omega$, we obtain
\begin{equation}\label{eq:v-estimate}
\begin{aligned}
\Delta v\le
-\frac{1}{6(n-1)}|\nabla A|^2+\frac{2n+1}{2n^2}|\nabla v|^2-\frac{2n^2-E}{12(n-1)}-n^2(2n-S).
\end{aligned}
\end{equation}

Suppose that $\inf_Mv<0$. By the Gauss equation, $\operatorname{Ric}\ge-(n+1)g$. Since $v$ is bounded below, the Omori--Yau maximum principle \cite{Omori1967,Yau1975} gives a sequence $\{p_j\}\subset M$ such that
\[
 v(p_j)\to\inf_Mv,\qquad |\nabla v|(p_j)\to0,\qquad
 \liminf_{j\to\infty}\Delta v(p_j)\ge0.
\]
For sufficiently large $j$, we have $v(p_j)<0$ and
\[
 E(p_j)>2n^2-\frac{2n-S(p_j)}{24(n-1)}
 \ge2n^2-\frac{n}{12(n-1)}>n^2.
\]
Hence $p_j\in\Omega$. By Lemma~\ref{lem:limit}, $E\le2n^2$, so $2n-S\ge-24(n-1)v$. It follows from \eqref{eq:v-estimate} that
\[
 \Delta v(p_j)\le\frac{2n+1}{2n^2}|\nabla v|^2(p_j)
 +24n^2(n-1)v(p_j).
\]
Letting $j\to\infty$, we obtain
\[
 0\le\liminf_{j\to\infty}\Delta v(p_j)
 \le24n^2(n-1)\inf_Mv<0,
\]
a contradiction. Therefore $v\ge0$.

If $v(p)=0$, then $p\in\Omega$. At this minimum, $\nabla v(p)=0$ and $\Delta v(p)\ge0$, so \eqref{eq:v-estimate} gives
\[
 0\le\Delta v(p)
 \le-\frac{|\nabla A|^2(p)}{6(n-1)}
 -\frac{2n^2-E(p)}{12(n-1)}-n^2(2n-S(p))\le0.
\]
Hence $E(p)=2n^2$ and $S(p)=2n$. Since $E\le2n^2$ and $E^n$ is subharmonic by \eqref{eq:power}, the strong maximum principle gives $E\equiv2n^2$, contrary to the hypothesis. Therefore $v>0$.

On $\Omega$, we have $2n^2-E\ge v$, so
\[
\Delta v\le\frac{2n+1}{2n^2}|\nabla v|^2-\frac{v}{12(n-1)}.
\]
Since $v>0$ on the compact support of $\varphi\in C_c^\infty(\Omega)$, multiplying the preceding inequality by $e^{-\frac{2n+1}{2n^2}v}\frac{\varphi^2}{v}$ and integrating by parts gives
\[
\begin{aligned}
\frac{1}{12(n-1)}\int_M e^{-\frac{2n+1}{2n^2}v}\varphi^2
&\le\int_M e^{-\frac{2n+1}{2n^2}v}
\left(2\frac{\varphi}{v}\langle\nabla v,\nabla\varphi\rangle
-\frac{\varphi^2}{v^2}|\nabla v|^2\right)\\
&\le\int_M e^{-\frac{2n+1}{2n^2}v}|\nabla\varphi|^2.
\end{aligned}
\]
Since $0<v\le2n^2-E<n^2$ on $\Omega$, the exponential factor lies between $e^{-\frac{2n+1}{2}}$ and $1$. This proves \eqref{eq:spectral}.
\end{proof}

\section{Proofs of the main theorems}\label{sec:main-proofs}

We first obtain local containment in a totally geodesic subsphere using Erbacher's method \cite[Section~2]{Erbacher1971} and then extend it to all of $M$ by unique continuation for height functions.

\begin{lemma}\label{lem:codimension-reduction}
Let $x:M^n\to\Sph^{n+q}$ be a minimal immersion of a connected manifold with flat normal bundle, where $n\ge2$ and $q\ge2$. If $E\equiv0$, then $x(M)$ is contained in a fixed totally geodesic subsphere $\Sph^{n+1}\subset\Sph^{n+q}$.
\end{lemma}
\begin{proof}
If $S\equiv0$, the immersion is totally geodesic and its image lies locally in a great $n$-sphere. Hence $x(U)\subset V$ for some nonempty open set $U\subset M$ and a fixed $(n+2)$-dimensional linear subspace $V\subset\R^{n+q+1}$.

If $S\not\equiv0$, choose a sufficiently small simply connected nonempty open set $U\subset\{S>0\}$. The condition $E=0$ implies that the first normal spaces on $U$ form a smooth line bundle. Choose an orthonormal normal frame on $U$ so that $N_1(p)=\Span\{\nu_1(p)\}$ for every $p\in U$. For $2\le\gamma\le q$, we have $A^\gamma=0$, and the covariant derivative formula gives $A^\gamma_{ij;k}=\omega^\perp_{1\gamma}(e_k)A^1_{ij}$. From Lemma~\ref{lem:symmetry}, we have
\[
 0=\sum_{i,j=1}^nA^\gamma_{ij}A^1_{ij;k}
  =\sum_{i,j=1}^nA^1_{ij}A^\gamma_{ij;k}
  =S\,\omega^\perp_{1\gamma}(e_k).
\]
Since $S>0$ on $U$, we have $\nabla^\perp\nu_1=0$.

By flatness of the normal bundle and simple connectedness of $U$, we may complete $\nu_1$ to a parallel orthonormal normal frame. For $2\le\gamma\le q$, the equalities $A^\gamma=0$ and $\nabla^\perp\nu_\gamma=0$, together with the Weingarten formula, give $\widetilde\nabla_X\nu_\gamma=0$ for every tangent vector field $X$. Thus $\nu_2,\ldots,\nu_q$ are constant orthonormal vectors in $\R^{n+q+1}$, and their common orthogonal complement is a fixed $(n+2)$-dimensional subspace $V$. The relations $\langle x,\nu_\gamma\rangle=0$ for $2\le\gamma\le q$ give $x(U)\subset V$.

In either case, for $a\in V^\perp$, the height function $u_a=\langle x,a\rangle$ vanishes on $U$. By minimality, $\Delta x=-nx$, so
\[
 (\Delta+n)u_a=\langle\Delta x+nx,a\rangle=0.
\]
By unique continuation, $u_a\equiv0$ on the connected manifold $M$. As $a\in V^\perp$ is arbitrary, $x(M)\subset V\cap\Sph^{n+q}$, a fixed totally geodesic $(n+1)$-sphere.
\end{proof}

\begin{proof}[Proof of Theorem~\ref{thm:main}]
Suppose that $E\not\equiv0$ and $E\not\equiv2n^2$. By Lemma~\ref{lem:limit}, there is a complete pointed smooth limit $M_\infty$ on which
\[
E_\infty\equiv2n^2,\qquad\operatorname{Ric}_\infty\ge0.
\]
Let $\varphi\in C_c^\infty(M_\infty)$. Using the convergence maps in Lemma~\ref{lem:compactness} on a relatively compact neighborhood of $\operatorname{supp}\varphi$ and extending by zero, we obtain $\varphi_j\in C_c^\infty(M)$. Since $E_\infty\equiv2n^2$, smooth convergence gives $\operatorname{supp}\varphi_j\subset\{E>n^2\}$ for sufficiently large $j$. Applying Lemma~\ref{lem:spectral} to $\varphi_j$ and letting $j\to\infty$, we obtain
\begin{equation}\label{eq:limit-spectral}
\lambda\int_{M_\infty}\varphi^2
\le\int_{M_\infty}|\nabla\varphi|^2.
\end{equation}

If $M_\infty$ is compact, taking $\varphi\equiv1$ in \eqref{eq:limit-spectral} gives a contradiction. If $M_\infty$ is noncompact, let $B_R$ denote the geodesic ball of radius $R$ centered at a fixed point. Choose a Lipschitz cutoff function $\varphi_R$ equal to $1$ on $B_R$ and supported in $B_{2R}$, with $|\nabla\varphi_R|\le\frac{2}{R}$. By smooth approximation and the Bishop--Gromov volume comparison theorem, \eqref{eq:limit-spectral} gives
\[
\lambda\le
\frac{\int_{M_\infty}|\nabla\varphi_R|^2}
{\int_{M_\infty}\varphi_R^2}
\le\frac{4}{R^2}
\frac{\operatorname{Vol}(B_{2R})}{\operatorname{Vol}(B_R)}
\le\frac{2^{n+2}}{R^2}.
\]
Letting $R\to\infty$ contradicts $\lambda>0$. Therefore $E\equiv0$ or $E\equiv2n^2$.

If $E\equiv0$, then $x(M)$ is contained in a fixed totally geodesic $\Sph^{n+1}$ by Lemma~\ref{lem:codimension-reduction}. It remains to consider $E\equiv2n^2$. Substituting into \eqref{eq:power}, we obtain
\[
 S\equiv2n,\qquad \nabla A\equiv0,\qquad
 \sum_{1\le\alpha<\beta<\gamma\le q}|A^\alpha|^2|A^\beta|^2|A^\gamma|^2=0.
\]
In a normal frame satisfying \eqref{eq:normal-orthogonality}, the vanishing triple sum gives $\rank\G\le2$. Since $E>0$, we have $\dim N_1=\rank\G=2$.

Pass to the universal cover $\widetilde M$ and use the same notation for the lifted immersion. The flat normal bundle over $\widetilde M$ admits a global parallel orthonormal frame. The corresponding shape operators are parallel and commute, so their joint eigenvalues are constant and their joint eigenspaces form parallel orthogonal distributions
\[
 T\widetilde M=\mathcal H_1\oplus\cdots\oplus\mathcal H_m.
\]
Write $n_a=\dim\mathcal H_a$, so $\sum_{a=1}^m n_a=n$. For each $a$, there is a parallel normal field $\eta_a$ such that, for $X\in T\widetilde M$ and $Y\in\mathcal H_a$,
\[
 A(X,Y)=\langle X,Y\rangle\eta_a.
\]
Thus $N_1=\Span\{\eta_1,\ldots,\eta_m\}$.
For unit vectors $X\in\mathcal H_a$ and $Y\in\mathcal H_b$ with $a\ne b$, parallelism gives $R(X,Y)Y\in\mathcal H_b$, and hence $\langle R(X,Y)Y,X\rangle=0$. The Gauss equation therefore gives $\langle\eta_a,\eta_b\rangle=-1$.

Since $\langle\eta_a,x\rangle=0$, we have
\[
 \langle\eta_a-x,\eta_b-x\rangle=0\quad(a\ne b),\qquad
 \langle\eta_a-x,x\rangle=-1.
\]
By minimality, $\sum_{a=1}^{m}n_a\eta_a=0$, and hence
\[
 \sum_{a=1}^{m}n_a(\eta_a-x)=-nx.
\]
Taking the inner product of the above equality with $\eta_a-x$, we obtain
\[
 n_a|\eta_a-x|^2=n.
\]
Moreover, $\Span\{\eta_1-x,\ldots,\eta_m-x\}=\R x\oplus N_1$. Since the vectors $\eta_a-x$ are nonzero and mutually orthogonal, $m=1+\dim N_1=3$.

For $1\le a\le3$, set
\[
 V_a(p)=\dd x(\mathcal H_a(p))\oplus\R(\eta_a(p)-x(p)).
\]
For local sections $X$ of $\mathcal H_b$ and $Y$ of $\mathcal H_a$, the Gauss and Weingarten formulas give
\begin{equation}\label{eq:endpoint-euclidean}
\begin{aligned}
 \widetilde\nabla_X\dd x(Y)
 &=\dd x(\nabla_XY)+\langle X,Y\rangle(\eta_a-x),\\
 \widetilde\nabla_X(\eta_a-x)
 &=-\left(\langle\eta_a,\eta_b\rangle+1\right)\dd x(X).
\end{aligned}
\end{equation}
Since $\mathcal H_a$ is parallel and $\langle\eta_a,\eta_b\rangle=-1$ for $a\ne b$, both derivatives in \eqref{eq:endpoint-euclidean} belong to $V_a$. Thus $V_a$ is parallel for the Euclidean connection and hence is a fixed linear subspace. The three subspaces are mutually orthogonal, with $\dim V_a=n_a+1$.

From the above identities, we obtain
\[
 x=\sum_{a=1}^{3}\frac{n_a}{n}(x-\eta_a),\qquad
 \frac{n_a}{n}(x-\eta_a)\in V_a,\qquad
 \left|\frac{n_a}{n}(x-\eta_a)\right|^2=\frac{n_a}{n}.
\]
Since the original immersion and its lift have the same image,
\[
 x(M)\subset\prod_{a=1}^{3}\Sph^{n_a}\!\left(\sqrt{\frac{n_a}{n}}\right)
 \subset V_1\oplus V_2\oplus V_3,
\]
where the $a$th sphere lies in $V_a$. The squared radii sum to $1$ and $\sum_{a=1}^3\dim V_a=n+3$, so the product lies in a fixed totally geodesic $\Sph^{n+2}$. Both $M$ and the product have dimension $n$ and carry the metrics induced from the ambient sphere, so $x$ is a local isometry into the product. Completeness of $M$ then implies that $x$ is a surjective Riemannian covering, proving alternative~\textup{(ii)}.
\end{proof}

\begin{proof}[Proof of Theorem~\ref{thm:lu-pinching}]
With $p$ as in Lemma~\ref{lem:lu-critical-derivative}, define
\[
 f(x)=
 \begin{cases}
 \frac{E(x)}{S(x)^p},& S(x)>0,\\
 0,& S(x)=0.
 \end{cases}
\]
The bound $0\le f\le S^{2-p}$, together with the pinching assumption, shows that $f$ is bounded and continuous on $M$, and it is smooth on $\{S>0\}$.

At a critical point where $f>0$, $\nabla f=0$ gives $\nabla E=\frac{pE}{S}\nabla S$, and hence
\[
\begin{aligned}
 \frac{\Delta f}{f}
 &=\frac{\Delta E}{E}-p\frac{\Delta S}{S}
   -\frac{2p\langle\nabla E,\nabla S\rangle}{ES}
   +p(p+1)\frac{|\nabla S|^2}{S^2}\\
 &=\frac{\Delta E}{E}-p\frac{\Delta S}{S}
   +p(1-p)\frac{|\nabla S|^2}{S^2}.
\end{aligned}
\]
From \eqref{eq:defect}, \eqref{eq:simons-flat}, and \eqref{eq:E-evolution}, we obtain
\[
\begin{aligned}
 \frac{\Delta f}{f}
 =&\frac{4}{E}\left[
 \sum_{\alpha=1}^q\left(S-|A^\alpha|^2-\frac{pE}{2S}\right)|\nabla A^\alpha|^2
 +\left(\frac{1}{2}+\frac{p(1-p)E}{4S^2}\right)|\nabla S|^2-2\sum_{\alpha,\beta=1}^q\sum_{k=1}^n
       \left(\sum_{i,j=1}^nA^\alpha_{ij}A^\beta_{ij;k}\right)^2\right]\\
 &+2\left[(2-p)n-(1-p)S-\frac{pE}{S}
       +\frac{6}{E}\sum_{1\le\alpha<\beta<\gamma\le q}
          \lambda_\alpha\lambda_\beta\lambda_\gamma\right].
\end{aligned}
\]
The first bracket is nonnegative by Lemma~\ref{lem:lu-critical-derivative}. Therefore
\[
 \frac{\Delta f}{f}
 \ge2\left[(2-p)n-(1-p)S-\frac{pE}{S}
       +\frac{6}{E}\sum_{1\le\alpha<\beta<\gamma\le q}
          \lambda_\alpha\lambda_\beta\lambda_\gamma\right].
\]
For the ordered eigenvalues in \eqref{eq:gram-eigenvalues}, \eqref{eq:defect} gives
\[
\begin{aligned}
 &3\sum_{1\le\alpha<\beta<\gamma\le q}\lambda_\alpha\lambda_\beta\lambda_\gamma
      -\frac{E}{2}\sum_{\alpha=3}^q\lambda_\alpha\\
 =&\sum_{\alpha=3}^q\lambda_\alpha
      \left[2\lambda_1\lambda_2-(\lambda_1+\lambda_2)\lambda_\alpha\right]+\sum_{3\le\alpha<\beta\le q}\lambda_\alpha\lambda_\beta
      \left(\lambda_1+\lambda_2-\lambda_\alpha-\lambda_\beta\right)\ge0.
\end{aligned}
\]
Every bracket is nonnegative by the ordering of the eigenvalues. When $q=2$, the sums involving indices at least three are empty. Moreover, by \eqref{eq:defect},
\[
 E\le S^2-\lambda_1^2
 \le2S(S-\lambda_1).
\]
Since $p\le\frac{2}{9}<\frac{1}{2}$, these two estimates yield
\begin{equation}\label{eq:lu-spectral-estimate}
\begin{aligned}
 \frac{6}{E}\sum_{1\le\alpha<\beta<\gamma\le q}
       \lambda_\alpha\lambda_\beta\lambda_\gamma-\frac{pE}{S}
 &\ge\sum_{\alpha=3}^q\lambda_\alpha-2p(S-\lambda_1)\\
 &=(1-2p)\sum_{\alpha=3}^q\lambda_\alpha-2p\lambda_2
 \ge-2p\lambda_2.
\end{aligned}
\end{equation}
Substituting this into the lower bound for $\frac{\Delta f}{f}$ gives
\begin{equation}\label{eq:lu-positive-critical}
 \frac{\Delta f}{f}
 \ge2\left[(2-p)n-(1-p)(S+\lambda_2)
                         +(1-3p)\lambda_2\right].
\end{equation}
From the pinching assumption, we have
\[
 S+\lambda_2\le2n+\frac{2n}{3n-2}
 =\frac{(2-p)n}{1-p}.
\]
Since $f>0$ implies $E>0$ and hence $\lambda_2>0$, while $1-3p>0$, we conclude that $\frac{\Delta f}{f}\ge2(1-3p)\lambda_2>0$.

Suppose that $\sup_M f>0$ and choose points $p_j$ with $f(p_j)\to\sup_M f$. The bound $f\le S^{2-p}$ implies, for all sufficiently large $j$,
\[
 S(p_j)\ge
 \left(\frac{1}{2}\sup_M f\right)^{\frac{1}{2-p}}>0.
\]
By the uniform bound on $S$ and Lemma~\ref{lem:compactness}, a subsequence of the immersions with base points $p_j$ converges smoothly to a complete pointed minimal immersion with base point $p_\infty$ and flat normal bundle. The ordered eigenvalues of $\G$ depend continuously on its entries, so the pinching inequality passes to the limit.

The lower bound for $S(p_j)$ gives $S_\infty(p_\infty)>0$. Thus $f_\infty=\frac{E_\infty}{S_\infty^p}$ is smooth on a neighborhood of $p_\infty$, where smooth convergence gives
\[
 f_\infty\le\sup_M f,\qquad
 f_\infty(p_\infty)=\sup_M f>0.
\]
Hence $f_\infty$ has a positive local maximum at $p_\infty$, so $\Delta f_\infty(p_\infty)\le0$. This contradicts \eqref{eq:lu-positive-critical}. Therefore $f\equiv0$ and $E\equiv0$, and Lemma~\ref{lem:codimension-reduction} completes the proof.
\end{proof}

\begin{remark}\label{rem:lu-weaker-pinching}
The conclusion of Theorem~\ref{thm:lu-pinching} remains valid if \eqref{eq:lu-pinching} is replaced by
\[
 S+\frac{4}{3n-2}\lambda_2\le2n+\frac{2n}{3n-2}.
\]
Indeed, at a point where $E>0$,
\[
 2S(S-\lambda_1)-E
 =(S-\lambda_1)^2+\sum_{\alpha=2}^q\lambda_\alpha^2>0.
\]
Take $p$ and $f$ as in the proof of Theorem~\ref{thm:lu-pinching}. Keeping the above remainder in the derivation of \eqref{eq:lu-spectral-estimate}, we obtain at a positive critical point of $f$
\[
\begin{aligned}
 \frac{\Delta f}{f}\ge2\left[(2-p)n-(1-p)S-2p\lambda_2\right]+\frac{2p}{S}
 \left[(S-\lambda_1)^2+\sum_{\alpha=2}^q\lambda_\alpha^2\right]>0.
\end{aligned}
\]
The first bracket is nonnegative by the assumed bound. This bound also gives a uniform upper bound for $S$ and passes to a smooth pointed limit. The same limit argument gives $E\equiv0$, and Lemma~\ref{lem:codimension-reduction} proves the conclusion.
\end{remark}

\section*{Acknowledgments}
The authors are grateful to Professors Hongwei Xu and Entao Zhao for insightful discussions and valuable suggestions on this work.

\end{document}